\pdfoutput=1
\documentclass[10pt,a4paper,oneside]{amsart}
\usepackage{amsmath, amssymb, amsfonts}
\usepackage[a4paper]{geometry}

\numberwithin{equation}{section}

\usepackage{stmaryrd}
\newcommand{\llb}{\llbracket}
\newcommand{\rrb}{\rrbracket}

\usepackage[scr=dutchcal]{mathalfa}

\usepackage{upgreek}
\renewcommand{\epsilon}{\varepsilon}
\renewcommand{\phi}{\varphi}
\renewcommand{\theta}{\vartheta}

\DeclareMathOperator{\aug}{aug}
\DeclareMathOperator{\Aut}{Aut}

\DeclareMathOperator{\diag}{diag}

\DeclareMathOperator{\Frac}{Frac}
\DeclareMathOperator{\Gal}{Gal}
\DeclareMathOperator{\GL}{GL}
\DeclareMathOperator{\gr}{gr}

\DeclareMathOperator{\id}{id}
\DeclareMathOperator{\Irr}{Irr}

\DeclareMathOperator{\ord}{ord}

\DeclareMathOperator{\res}{res}
\DeclareMathOperator{\Quot}{Quot}
\newcommand{\SK}{S\!K}
\DeclareMathOperator{\Tr}{Tr}

\newcommand{\FFF}{\mathfrak F}
\newcommand{\KKK}{\mathfrak K}
\newcommand{\LLL}{\mathfrak L}
\newcommand{\mm}{\mathfrak m}
\newcommand{\pp}{\mathfrak p}

\newcommand{\G}{\mathcal G}
\newcommand{\OO}{\mathcal O}
\newcommand{\Q}{\mathcal Q}

\newcommand{\CC}{\mathbb C}
\newcommand{\FF}{\mathbb F}

\newcommand{\QQ}{\mathbb Q}
\newcommand{\RR}{\mathbb R}
\newcommand{\ZZ}{\mathbb Z}

\newcommand{\al}{\mathrm{c}}
\newcommand{\cent}{\mathfrak z}

\newcommand{\Dchi}{\mathfrak D}
\newcommand{\Emax}{\mathfrak E}
\newcommand{\pc}{\delta_\gamma}

\newcommand{\pt}{\delta_\tau}

\newcommand{\gcp}{\gamma^{\prime\prime}_\eta}
\newcommand{\gcpf}{\gamma^{ F,\prime\prime}_\eta}

\newcommand{\Gcp}{\Gamma^{\prime\prime}_\eta}

\newcommand{\hD}{\widehat D}
\newcommand{\Dcp}{\widehat{\Dchi}_\pp}
\newcommand{\Sigmachi}{\mathfrak O}
\newcommand{\sss}{\nu_{\chi}}

\newcommand{\tD}{\widetilde D}

\newcommand{\tauf}{\tau}
\newcommand{\vf}{v_\chi}
\newcommand{\af}{a}
\newcommand{\DFchi}{D_\chi}
\newcommand{\nFchi}{n_\chi}
\newcommand{\sFchi}{s_\chi}
\newcommand{\nFeta}{n_\eta}
\newcommand{\sFeta}{s_\eta}

\newcommand{\yFeta}{y_\eta}
\newcommand{\YFeta}{Y_\eta}

\usepackage{tikz-cd}
\usepackage{colonequals}

\newenvironment{psmallmatrix}
{\left(\begin{smallmatrix}}
	{\end{smallmatrix}\right)}

\usepackage{tikz}
\usetikzlibrary{cd}
\usetikzlibrary{arrows}
\usetikzlibrary{positioning}
\usetikzlibrary{calc}

\usepackage[style=alphabetic,backend=biber,isbn=false]{biblatex}
\usepackage{enumitem}
\newlist{theoremlist}{enumerate}{1}
\setlist[theoremlist]{label=(\roman{theoremlisti}), ref=\thetheorem.\roman{theoremlisti},noitemsep, topsep=.2ex}
\newlist{propositionlist}{enumerate}{1}
\setlist[propositionlist]{label={(\roman{propositionlisti})}, ref=\theproposition.\roman{propositionlisti},noitemsep, topsep=.2ex}
\newlist{conjecturelist}{enumerate}{1}
\setlist[conjecturelist]{label={(\roman{conjecturelisti})}, ref=\theconjecture.\roman{conjecturelisti},noitemsep, topsep=.2ex}
\newlist{lemmalist}{enumerate}{1}
\setlist[lemmalist]{label=(\roman{lemmalisti}), ref=\thelemma.\roman{lemmalisti},noitemsep, topsep=.2ex}

\usepackage{etoolbox} 

\usepackage{mathtools}
\let\underbrace\relax
\let\underbrace\LaTeXunderbrace

\usepackage[hidelinks,plainpages=false,pdfpagelabels]{hyperref}
\usepackage[capitalise, noabbrev]{cleveref}

\theoremstyle{plain}
\newtheorem{theorem}{Theorem}[section]
\newtheorem{lemma}[theorem]{Lemma}

\newtheorem{proposition}[theorem]{Proposition}
\newtheorem{corollary}[theorem]{Corollary}

\newtheorem*{theorem*}{Theorem}
\newtheorem*{conjecture*}{Conjecture}
\theoremstyle{definition}
\newtheorem{definition}[theorem]{Definition}
\theoremstyle{remark}
\newtheorem{remark}[theorem]{Remark}

\Crefname{lemma}{Lemma}{Lemmata}
\Crefname{claim}{Claim}{Claims}
\Crefname{proposition}{Proposition}{Propositions}
\Crefname{conjecture}{Conjecture}{Conjectures}
\Crefname{example}{Example}{Examples}
\crefname{page}{page}{pages}
\Crefname{condition}{Condition}{Conditions}
\Crefname{question}{Question}{Questions}
\Crefname{theoremlisti}{Theorem}{Theorems}
\Crefname{propositionlisti}{Proposition}{Propositions}
\Crefname{conjecturelisti}{Conjecture}{Conjectures}
\Crefname{lemmalisti}{Lemma}{Lemmata}

\usepackage{scalerel,stackengine}
\stackMath
\newcommand\reallywidehat[1]{%
	\savestack{\tmpbox}{\stretchto{%
			\scaleto{%
				\scalerel*[\widthof{\ensuremath{#1}}]{\kern-.6pt\bigwedge\kern-.6pt}%
				{\rule[-\textheight/2]{1ex}{\textheight}}
			}{\textheight}%
		}{0.5ex}}%
	\stackon[1pt]{#1}{\tmpbox}%
}

\NewDocumentEnvironment{noproof}{m}{
  \par\pushQED{\qed}\UseName{#1}%
}{\popQED\UseName{end#1}}

\usepackage{bbm}

\title[On the Wedderburn decomposition of $\Q^ F(\G)$]{On the Wedderburn decomposition of the total ring of quotients of certain Iwasawa algebras}
\author{Ben Forrás}
\address{University of Ottawa\\
	Department of Mathematics and Statistics \\
    STEM Complex \\
	150 Louis-Pasteur Pvt\\
    Ottawa, ON \\
	Canada K1N 6N5}
\email{bforras@uottawa.ca}
\urladdr{https://bforras.eu}

\subjclass[2020]{16H10, 11R23} 
\keywords{Wedderburn~decomposition, Iwasawa algebra, skew~power~series~ring}
\date{Version of 2025-12-16}

\begin{document}

\begin{abstract}
Let $\G\simeq H\rtimes\Gamma$ be the semidirect product of a finite group $H$ and $\Gamma\simeq\ZZ_p$. Let $ F/\QQ_p$ be a finite extension with ring of integers $\OO_ F$. Then the total ring of quotients $\Q^ F(\G)$ of the completed group ring $\OO_{ F}\llb\G\rrb$ is a semisimple ring. 
We determine its Wedderburn decomposition {under a ramification hypothesis} by relating it to the Wedderburn decomposition of the group ring $ F[H]$.
\end{abstract}

\maketitle

\section*{Introduction}
{Let $p$ be an odd prime number.} Let $\G= H\rtimes \Gamma$ be a profinite group that can be written as the semidirect product of a finite normal subgroup $H$ and a pro-$p$ group $\Gamma\simeq\ZZ_p$ isomorphic to the additive group of the $p$-adic integers. (Then $\G$ is a $1$-dimensional $p$-adic Lie group.) Let $ F$ be a finite extension of $\QQ_p$ with ring of integers $\OO_ F$, and consider the completed group ring
$\Lambda^{\OO_ F}(\G)\colonequals \OO_ F\llb \G\rrb$.
Let $n_0$ be a large enough integer such that $\Gamma_0\colonequals \Gamma^{p^{n_0}}$ is central in $\G$. Let
$\Q^ F(\G)\colonequals \Quot(\Lambda^{\OO_ F}(\G))$ be the total ring of quotients of $\Lambda^{\OO_ F}(\G)$.

The ring $\Q^{ F}(\G)$ is semisimple artinian, a fact due to Ritter and Weiss \cite{TEIT-II}; it admits a Wedderburn decomposition
\[\Q^ F(\G) \simeq \bigoplus_{\chi\in \Irr(\G)/\sim_ F} M_{\nFchi}(\DFchi)  .\]
Here $\Irr(\G)$ denotes the set of irreducible characters of $\G$ with open kernel, and the equivalence relation $\sim_ F$ on $\Irr(\G)$ is defined as follows: two characters $\chi,\chi'$ are equivalent if there is a $\sigma\in\Gal( F_\chi/ F)$ such that ${}^\sigma(\res^\G_H \chi)=\res^\G_H \chi'$ where $ F_\chi= F(\chi(h):h\in H)$. For each equivalence class, we have a skew field $\DFchi$.

The aim of this paper is determining the Wedderburn decomposition of $\Q^ F(\G)$, that is, describing the skew field $\DFchi$, its Schur index $\sFchi$, and the size $\nFchi$ of the corresponding matrix ring.

If $\G=H\times\Gamma$ is a \emph{direct} product, then this is a trivial task: the Wedderburn decomposition of $\Q^ F(\G)$ is directly determined by that of the group ring $ F[H]$. Indeed, write
\[ F[H]\simeq\bigoplus_{\eta\in\Irr(H)/\sim_{ F}} M_{\nFeta}(D_\eta)\]
for this Wedderburn decomposition. Then since $\Lambda^{\OO_ F}(\G)= \Lambda^{\OO_ F}(\Gamma)[H]$, we get that
\[\Q^ F(\G)\simeq\bigoplus_{\eta\in\Irr(H)/\sim_{ F}} M_{\nFeta}\left(\Q^ F(\Gamma)\otimes_{ F} D_\eta\right) .\]
The relationship with the decomposition above is given by $\eta=\res^\G_H\chi$, see \cite[Theorem~4.21]{Isaacs}.

The semidirect product case is significantly more difficult. A first step was taken by Lau \cite[Theorem~1]{Lau}, who tackled the case when $\G$ is pro-$p$, that is, when $H$ is a finite $p$-group. Her methods make use of a theorem of Schilling also attributed to Witt and Roquette \cite[(74.15)]{CRII}, which states that in this case, the Schur indices $\sFeta$ of the skew fields $D_\eta$ are all $1$: in other words, each $D_\eta$ is a field. This is not true for general $\G$. Note that the skew fields $D_\chi$ may still have nontrivial Schur indices even when $\G$ is pro-$p$: an example of this phenomenon was computed by Lau \cite[p. 1232ff]{Lau}.

Nickel made strides towards obtaining results in the case of general $\G$ in \cite[\S1]{NickelConductor}. He described the centre of the skew field $\DFchi$, provided a sufficient criterion for a field to be a splitting field thereof, and proved the following divisibilities:
\begin{align*}
	\sFchi &\mid \sFeta ( F(\eta): F_{\chi})  , \\
	\nFeta &\mid \nFchi  .
\end{align*}
Here {$\eta\mid\res^\G_H\chi$ is an irreducible constituent, and} $ F(\eta)= F(\eta(h):h\in H)$: this is a finite extension of $ F_\chi$.

The present work builds upon both of {these} approaches: we generalise Lau's description of the skew fields $\DFchi$, and determine the missing factors in Nickel's divisibilities. We now describe our results.

Let us fix a topological generator $\gamma$ of $\Gamma$. We let $\Gamma$ and $\Gal( F(\eta)/ F)$ act on $\Irr(H)$: for $\eta\in\Irr(H)$, let ${}^{\gamma}\eta(h)=\eta(\gamma h\gamma^{-1})$ and ${}^\sigma\eta(h)=\sigma(\eta(h))$ for all $h\in H$ and $\sigma\in \Gal( F(\eta)/ F)$.
These two actions are related as follows: for $\chi\in\Irr(\G)$ and $\eta\mid\res^\G_H\chi$ an irreducible constituent of its restriction to $H$, we define $\vf$ to be the minimal positive exponent such that $\gamma^{\vf}$ acts as a Galois automorphism of $ F(\eta)/ F$ on $\eta$; it can be shown that $\vf$ only depends on $\chi$. There is a unique automorphism $\tauf\in \Gal( F(\eta)/ F)$ such that 
\[{}^{\gamma^{\vf}}\eta={}^{\tauf}\eta  .\] 
Furthermore, this $\tauf$ fixes $ F_\chi$, {and} generates $\Gal( F(\eta)/ F_{\chi})$. {These assertions will be verified in \cref{Qpeta-Qpchi-cyclic}. Moreover, $\tau$} admits a unique extension as an automorphism of $D_\eta$ of the same order {(\cref{extending-tau-to-D}), which we will also denote by $\tau$.}

Our main result is the following: {
\begin{theorem*} 
	Let $ F$ be a finite extension of $\QQ_p$. {Let $\chi\in \Irr(\G)$ and $\eta\mid \res^\G_H\chi$ be an irreducible constituent. Suppose that the extension $F(\eta)/F_\chi$ is totally ramified.} Then
	\begin{theoremlist}
		\item $\nFchi = \nFeta \vf$,
		\item $\sFchi = \sFeta (F(\eta):F_\chi)$,
		\item $\DFchi \simeq \Quot\left(\OO_{D_\eta}[[X; \tau, \tau-\id]]\right)$.
	\end{theoremlist}
\end{theorem*}
The ring {$\OO_{D_\eta}[[X; \tau, \tau-\id]]$} is the skew power series ring whose underlying additive group agrees with that of the power series ring $\OO_{D_\eta}[[X]]$, with multiplication rule {$Xd=\tau(d)X+(\tau-\id)(d)$} for all $d\in\OO_{D_\eta}$. Such skew power series rings were studied by Schneider and Venjakob \cite{VenjakobWPT,SchneiderVenjakob}.}

{A comment on the condition that $F(\eta)/F_\chi$ be totally ramified is in order. This hypothesis is satisfied in all cases in which the Wedderburn decomposition of $\Q^F(\G)$ had previously been determined: see \cref{rmk:totally-ramified-hypothesis} for details. For the technical details of where this assumption is used, we refer to \cref{rmk:total-ramification-technical}. We shall address the case of arbitrary ramification in forthcoming work.}

The material is organised as follows. \Cref{sec:preliminaries} consists of a collection of background in preexisting work. \Cref{sec:extending-Galois-Deta} details the process of extending a Galois automorphism of a local field to a skew field with this centre. \Cref{sec:skew-power-series} describes the skew power series rings appearing in our main result. \Cref{sec:Wedderburn} contains a study of the extension $ F(\eta)/ F_\chi$ and the precise formulation of the statements above. \Cref{sec:actions} relates the Galois action of $\Gal( F(\eta)/ F_\chi)$ to the group action of $\Gamma$ on $\Irr(H)$. {We prove our main result in \cref{sec:TORC}.}

This paper contains results obtained in Chapters 2, 3 and 5 of the author's doctoral thesis \cite{thesis}.

\subsection*{Acknowledgements} The author wishes to extend his heartfelt gratitude to Andreas Nickel for his guidance and support throughout his doctoral studies, as well as for his many helpful comments on a draft version of this article. He also thanks Alessandro Cobbe for his valuable observations. {Furthermore, he is grateful to the anonymous referee for their very thorough reports and numerous constructive suggestions.}

The results in the aforementioned thesis were obtained while the author was a member of the Research Training Group 2553 at the Essen Seminar for Algebraic Geometry and Arithmetic (ESAGA), funded by the Deutsche Forschungsgemeinschaft (DFG, project no.~412744520).

\subsection*{Notation and conventions} The letter $p$ always stands for an odd prime number.

The word `ring' is short for `not necessarily commutative ring with unity'. A domain is a ring with no zero divisors. If $R$ is a ring, $\Quot(R)$ stands for the total ring of quotients of $R$, which is obtained from $R$ by inverting all central regular elements. When $R$ is an integral domain (i.e. commutative), this is the field of fractions, and to emphasise this, we use the notation $\Frac(R)$ instead.

We will abuse notation by writing $\oplus$ for a direct product of rings, even though this is not a coproduct in the category of rings.
The centre of a group $G$ resp. a ring $R$ is denoted by $\cent(G)$ resp. $\cent(R)$.
An algebraic closure of a field $ F$ is denoted by $ F^\al$.
Overline means either topological closure or residue (skew) field, but never algebraic closure.
For $n\ge1$, $\mu_n$ stands for the group of $n$th roots of unity.

In our notation, we make a clear distinction between rings of (formal) power series and completed group algebras: we use double square brackets for the former and blackboard square brackets for the latter, so $\ZZ_p[[T]]$ is a ring of power series, and $\ZZ_p\llb \Gamma\rrb$ is a completed group algebra.

In \cref{sec:Wedderburn,sec:actions,sec:TORC}, we will fix a base field $ F$. Most of the objects at hand depend on the choice of this field. This is not always reflected in our notation: in order to prevent it from becoming too cumbersome, the field $ F$ is usually suppressed from it. If the choice of the base field is particularly relevant, for example because we are comparing objects coming from different base fields, we attach the base as a superscript to our notation, e.g. writing $D_\chi^ F$ instead of $D_\chi$.

\section{Algebraic preliminaries} \label{sec:preliminaries}
Let $\G$ be a profinite group. For $ F/\QQ_p$ a finite field extension with ring of integers $\OO_ F$, define the Iwasawa algebra of $\G$ over $\OO_ F$ as
\[\Lambda^{\OO_ F}(\G)\colonequals \OO_ F\otimes_{\ZZ_p} \Lambda(\G)\colonequals \OO_ F\otimes_{\ZZ_p}\ZZ_p\llb \G\rrb=\OO_ F\llb\G\rrb  .\]
Let $\Q^ F(\G)\colonequals\Quot(\OO_ F\llb\G\rrb)$ denote its total ring of quotients. Note that {$\Q^ F(\G)= F\otimes_{\QQ_p}\Q^{\QQ_p}(\G)$} by \cite[Lemma~1]{TEIT-II}.

From now on, let $\G=H\rtimes\Gamma$, where $H$ is a finite group and $\Gamma\simeq\ZZ_p$ is isomorphic to the additive group of the $p$-adic integers.
Let $\Gamma_0\colonequals\Gamma^{p^{n_0}}$ where $n_0$ is chosen to be large enough such that $\Gamma_0\subseteq\G$ is central.

As in \cite[\S4.2]{NABS} and in the proof of \cite[Proposition~5]{TEIT-II}, the Iwasawa algebra resp. its total ring of quotients  admit the following decompositions:
\begin{align}
	\Lambda^{\OO_ F}(\G)=\bigoplus_{i=0}^{p^{n_0}-1} \Lambda^{\OO_ F}(\Gamma_0)[H] \gamma^i  , \quad 
	\Q^{ F}(\G)=\bigoplus_{i=0}^{p^{n_0}-1} \Q^{ F}(\Gamma_0)[H] \gamma^i  . \label{eq:Q-decomposition}
\end{align}
Moreover, $\Q^{ F}(\G)$ is a semisimple artinian ring, as shown in \cite[Proposition~5(1)]{TEIT-II}. Our goal is to determine its Wedderburn decomposition. We begin by considering characters and their associated idempotents as follows. By a complex resp. $p$-adic Artin character we shall mean the trace of a Galois representation on a finite dimensional vector space over $\CC$ resp. $\QQ_p^\al$ with open kernel.
Let $\Irr(\G)$ denote the set of absolutely irreducible $p$-adic Artin characters, and let $\chi\in\Irr(\G)$.

Let $\eta\mid\res^\G_H\chi$ be an irreducible constituent of the restriction of $\chi$ to $H$.
{There is a left action of $\G$ on the set of irreducible constituents of $\res^\G_H\chi$ given by ${}^g\eta(h)=\eta(g^{-1} h g)$ for $h\in H$.}
Write $w_\chi\colonequals(\G:\G_\eta)$ where $\G_\eta$ is the stabiliser of $\eta$ in $\G$; this is a power of $p$ since $H$ stabilises $\eta$, and it can be shown that $w_\chi$ only depends on $\chi$.
{Compare the action just defined with the left action of $\Gamma$ defined in the Introduction as ${}^\gamma\eta(h)=\eta(\gamma h\gamma^{-1})$ for $h\in H$. As $\Gamma$ is abelian (as opposed to $\G$), this is also a left action. From now on, we will always consider the action of $\Gamma$ or $\G/\G_\eta$, which is isomorphic to a quotient of $\Gamma$. Later in our computations this will provide the appropriate formula to move powers of $\gamma$ to the right (rather than the left).}

Let $ F_{\chi} \colonequals  F(\chi(h): h\in H)$; this is contained in $ F(\eta)$. The field $ F(\eta)$ is abelian over $ F$, as it is contained in some cyclotomic extension, hence the extension $ F(\eta)/ F_\chi$ is Galois.

Using Clifford theory, one can show -- see \cite[Lemma 1.1]{NickelConductor} -- that there are irreducible constituents $\eta_1,\ldots,\eta_{\sss}$ of the restriction $\res^\G_H\chi$ such that there is a decomposition
\begin{equation} \label{eta-i}
	\res^\G_H\chi = \sum_{i=1}^{\sss} \sum_{\sigma\in\Gal( F(\eta_i)/ F_{\chi})} {}^\sigma \eta_i = \sum_{g\in\G/\G_\eta}{}^g\eta = \sum_{i=0}^{w_\chi-1} {}^{\gamma^i}\eta  .
\end{equation}
This shows that irreducible constituents of $\res^\G_H\chi$ are all $\G$-conjugates of one another. Moreover, it follows that the index $w_\chi$ depends only on $\chi$. Since $H$ is a normal subgroup, this shows that for a given $\chi$, the field $ F(\eta)$ does not depend on the choice of $\eta$. In particular, all degrees $( F(\eta_i): F_\chi)$ are equal. It follows readily that the number $\sss$ satisfies
\begin{equation} \label{wchi-s}
	w_\chi=\sss\cdot( F(\eta): F_{\chi})
\end{equation}
In particular, the degree $( F(\eta): F_{\chi})$ is a power of $p$ because $w_\chi$ is. Note that $w_\chi$ is independent of $ F$. 
The number $\sss$ depends on $\chi$ (and $ F$), but for $\chi$ (and $ F$) fixed, it is independent of the choice of $\eta$, since both $w_\chi$ and $( F(\eta): F_{\chi})$ are.

As in \cite{TEIT-II}, we have the following idempotents:
\begingroup\allowdisplaybreaks
\begin{align*}
	e(\eta) &\colonequals \frac{\eta(1)}{\#H} \sum_{h\in H} \eta(h^{-1}) h \in  F(\eta)[H]  , &
	e_\chi &\colonequals \sum_{g\in\G/\G_\eta} e({}^g\eta) \in  F(\eta)[H]  .
\end{align*}\endgroup
Following \cite[(3)]{NickelConductor}, we also wish to consider analogous idempotents with $ F$-coefficients; this is achieved by taking the respective Galois orbits. For this, note that \eqref{eta-i} implies
\[e_\chi = \frac{\chi(1)}{\# H w_\chi} \sum_{h\in H} \chi(h^{-1}) h \in  F_{\chi}[H], \]
showing that $e_\chi$ in fact has coefficients in $ F_\chi$.
The respective idempotents over $ F$ are then defined as follows.
\begin{align} \label{eq:def-epsilon}
	\epsilon(\eta) &\colonequals \sum_{\sigma\in\Gal( F(\eta)/ F)} e({}^\sigma\eta) \in  F[H]  , &
	\epsilon_\chi &\colonequals \sum_{\sigma\in\Gal( F_{\chi}/ F)} \sigma(e_\chi) \in  F[H]  .
\end{align}
Ritter and Weiss showed \cite[556]{TEIT-II} that every primitive central idempotent of {$\QQ_p^\al\otimes_{\QQ_p}\Q(\G)$} is of the form $e_\chi$; it follows that primitive central idempotents of $\Q^ F(\G)$ are of the form $\epsilon_\chi$.
Finally, we note the following consequence of \eqref{eta-i}:
\begin{align} \label{eq:epsilon-chi-etas}
	\epsilon_\chi &= {\sum_{i=1}^{\sss}} \epsilon(\eta_i)  .
\end{align}

\section{Extending Galois action to skew fields} \label{sec:extending-Galois-Deta}

This section is concerned with the question of extending a Galois automorphism of the centre of a finite dimensional skew field over a local field to the entire skew field, under certain assumptions.

Let $ K$ be a local field. Fix a uniformiser $\pi_ K$ of $ K$, and let $\FF_ q$ be the residue field of $ K$, with $ q$ a power of $p$. Let $ D$ be a skew field with centre $ K$ and index $s$: this admits the following description as a cyclic algebra, see \cite[§14]{MO}.
Let $\omega$ be a primitive $( q^{s}-1)$st root of unity, and let $\sigma\in\Gal( K(\omega)/ K)$ be the Galois automorphism defined by $\sigma(\omega)\colonequals\omega^{ q^{ r}}$, where $ r/s$ is the Hasse invariant of $ D$. Note that $\sigma$ is a generator of this Galois group. Then
\begin{equation} \label{eq:D-cyclic-algebra}
	 D=\big(  K(\omega) /  K, \sigma, \pi_ K \big)=\bigoplus_{i=0}^{s-1}  K(\omega) \pi_{ D}^i  ,
\end{equation}
where $\pi_{ D}^{s}=\pi_ K$ and $\pi_{ D}\omega=\sigma(\omega)\pi_{ D}$.

Let $ K/ k$ be a Galois $p$-extension of local fields, and let $\tau\in\Gal( K/ k)$ be a Galois automorphism of order $d$ (necessarily a $p$-power).
{Let $K^{\langle\tau\rangle}$ be the subfield of $K$ fixed by $\tau$, and write $q_\tau$ for the order of its residue field.}
Furthermore, assume that the index $s\mid( q_\tau-1)$; in particular, $s$ is coprime to $p$.
Our objective is to extend $\tau$ to an automorphism of $ D$ with the same order $d$ as $\tau$: this will be accomplished in \cref{extending-tau-to-D}.

\begin{lemma}
	The extension $ K(\pi_{ D})/ K$ is a cyclic Galois extension, totally ramified of degree $s$.
\end{lemma}
\begin{proof}
	Indeed, $\pi_{ D}$ has minimal polynomial $X^{s}-\pi_ K\in K[X]$, which is Eisenstein, hence the extension is totally ramified. Let $\zeta_{s}$ be a primitive $s$th root of unity: this is contained in $ K^{\langle\tau\rangle}\subseteq K$ because $s\mid ( q_\tau-1)$. The roots of $X^{s}-\pi_ K$ are of the form $\zeta_{s}^i \pi_{ D}$, which are therefore all contained in $ K(\pi_{ D})$. Hence the extension is Galois, and a generator of the Galois group is given by {$\pi_ D\mapsto \zeta_s\pi_ D$}.
\end{proof}

The automorphism $\tau$ preserves the valuation of $ K$, wherefore $\tau(\pi_ K)=\epsilon\cdot\pi_ K$ for some $\epsilon\in\OO_{ K}^\times$. Under the decomposition $\OO_{ K}^\times=\mu_{ q-1}\times U_{ K}^1$, we can write $\epsilon=\zeta\cdot u$ where $\zeta\in\mu_{ q-1}$ and $u\in U_{ K}^1$, where $U^1_ K=\{u\in \OO_{ K}^\times:u\equiv 1\pmod{\pi_ K}\}$ is the group of $1$-units. Let us write $N_{\langle\tau\rangle}$ for the norm map from $ K$ to $ K^{\langle\tau\rangle}$.
\begin{lemma} \label{epsilon-has-N-1} \label{zeta-u-have-N-1}
	The elements $\epsilon$, $\zeta$ and $u$ all have norm $1$ in $ K^{\langle\tau\rangle}$:
	\[1=N_{\langle\tau\rangle}(\epsilon)=N_{\langle\tau\rangle}(\zeta) = N_{\langle\tau\rangle}(u)  .\]
\end{lemma}
\begin{proof}
	Apply $N_{\langle\tau\rangle}$ to $\tau(\pi_ K)=\epsilon\cdot\pi_ K$: it follows that
	$1=N_{\langle\tau\rangle}(\epsilon)=N_{\langle\tau\rangle}(\zeta) \cdot N_{\langle\tau\rangle}(u).$
	Since $u$ is a $1$-unit, so is its norm, that is, $N_{\langle\tau\rangle}(u)\equiv 1\pmod{\pi_\tau}$, {where $\pi_\tau$ is a uniformiser of $K^{\langle\tau\rangle}$}. Then $N_{\langle\tau\rangle}(\zeta)$ must also be a $1$-unit. On the other hand, $\zeta$ is a root of unity of order prime to $p$, and thus so is its norm. The groups of $1$-units resp. roots of unity in $\OO_{ K^{\langle\tau\rangle}}^\times$ have intersection the $p$-power roots of unity. Therefore $N_{\langle\tau\rangle}(\zeta)=1$, which forces $N_{\langle\tau\rangle}(u)=1$ as well.
\end{proof}

\begin{lemma} \label{zeta-in-mu-smaller}
	The element $\zeta\in\mu_{ q-1}$ has order dividing $\frac{ q-1}{ q_\tau-1}$, that is, 
	$\zeta\in\mu_{( q-1)/( q_\tau-1)}$.
\end{lemma}

\begin{proof}
	By definition, the cokernel of the norm map $N_{\langle\tau\rangle}$ on $\mu_{ q-1}$ is the $0$th Tate cohomology group $\hat H^0\left(\langle \tau\rangle, \mu_{ q-1} \right) .$
	Since $\langle \tau\rangle$ is a $p$-group, and $\mu_{ q-1}$ has order prime to $p$, this cohomology group vanishes, so $N_{\langle\tau\rangle}$ is surjective onto {$\mu_{q_\tau-1}$, and thus} $\ker N_{\langle\tau\rangle}$ has order $\# \mu_{ q-1} / \# \mu_{ q_\tau-1} .$
	Since $\mu_{ q-1}$ is a cyclic group, the kernel is the unique subgroup of this order, which is $\mu_{( q-1)/( q_\tau-1)}$. \Cref{zeta-u-have-N-1} finishes the proof.
\end{proof}

\begin{lemma} \label{extending-tau-non-uniquely}
	The Galois automorphism $\tau\in\Gal( K/ k)$ admits an extension to an {automorphism $\tilde \tau\in \Aut_{ k}\left( K(\pi_{ D})\right)$. Moreover, there is a unit $\epsilon_D\in\OO_K^\times$ such that $\tilde\tau(\pi_D)=\epsilon_D\pi_D$.}
\end{lemma}
\begin{proof}
	For valuation reasons, {$\tilde\tau(\pi_{ D})$} must be of the form $\epsilon_ D \cdot \pi_{ D}$ for some $\epsilon_ D\in \OO_{ K(\pi_{ D})}^\times$, and it must satisfy {$\tilde\tau(\pi_{ D})^{s}=\tilde\tau(\pi_{ D}^{s})$}. The latter is equivalent to requiring $\epsilon_ D^{s} \cdot \pi_{ K} = \epsilon\cdot \pi_{ K}$, that is, $\epsilon_ D^{s}=\epsilon$. So extending $\tau$ {to a $k$-automorphism $\tilde\tau$} of $ K(\pi_{ D})$ means finding an $s$th root of $\epsilon$.
	
	Since $s$ is coprime to $p$ by assumption, the $s$th-power-map is bijective on $1$-units, so $u$ has an $s$th root. Moreover, we know that $\zeta\in\mu_{( q-1)/( q_\tau-1)}$ and that $s\mid q_\tau-1$, whence $\zeta\in\mu_{( q-1)/s}$, which implies that $\zeta$ also has an $s$th root. So such an $\epsilon_D$ exists, and in fact, this argument shows that it is contained in $ K$.
\end{proof}

The extension {$\tilde\tau$} of $\tau$ to $ K(\pi_{ D})$ as given in \cref{extending-tau-non-uniquely} is not unique due to the choice of $\epsilon_D$. It becomes unique under an additional assumption on the order:

\begin{lemma} \label{extending-tau-uniquely}
	The Galois automorphism $\tau\in\Gal( K/ k)$ admits a unique extension to an element $\tilde\tau\in\Aut_{ k}\left( K(\pi_{ D})\right)$ of order $d$.
\end{lemma}
\begin{proof}
	Since $\tilde\tau^{d} (\pi_{ D})=N_{\langle\tau\rangle}(\epsilon_{ D})\pi_ D$, such an extension has order $d$ if and only if $N_{\langle\tau\rangle}(\epsilon_{ D}) = 1$.
	
	We first show uniqueness. Suppose $\epsilon_{ D},\epsilon_{ D}'\in\OO_ K^\times$ are both $s$th roots of $\epsilon$ with $\tau$-norm $1$. Then $(\epsilon_{ D}'/\epsilon_ D)^s=1$ implies $\epsilon_{ D}'=\xi \epsilon_{ D}$ for some $\xi\in\mu_s( K)$, and it follows that $N_{\langle\tau\rangle}(\xi)=1$. 
	Since $s\mid( q_\tau-1)$, we have $\mu_s( K)=\mu_s( K^{\langle\tau\rangle})$. In particular, $\xi\in  K^{\langle\tau\rangle}$, and so $N_{\langle\tau\rangle}(\xi)=\xi^d$. The $d$th power map is bijective on $\mu_s( K)$, hence $\xi=1$, which proves uniqueness.
	
	Now we turn to existence. Let $\epsilon_{ D}\in\OO_ K^\times$ be an element as in \cref{extending-tau-non-uniquely}. From \cref{epsilon-has-N-1} we have $1=N_{\langle\tau\rangle}(\epsilon)=N_{\langle\tau\rangle}(\epsilon_ D)^{s}$, so $N_{\langle\tau\rangle}(\epsilon_ D)=\zeta_{s}^a$ for some $a$. Since $d$ is coprime to $s$, the $d$th power map is bijective on $\mu_s( K)$, and so there exists a unique $s$th root of unity $\tilde\zeta\in\mu_s( K^{\langle\tau\rangle})$ such that $\tilde\zeta^{d}=\zeta_{s}^a$. Then $N_{\langle\tau\rangle}\!\left(\tilde\zeta^{-1} \cdot\epsilon_ D\right)= \tilde\zeta^{-d} \cdot \zeta_{s}^a = 1.$
	Replacing the given $\epsilon_ D$ by $\tilde\zeta^{-1} \cdot\epsilon_ D$, we get the unique root of $\epsilon$ which makes the norm condition hold.
\end{proof}

{From now on, $\tilde\tau$ denotes the unique degree $d$ extension of $\tau$ from \cref{extending-tau-uniquely}.}

\begin{lemma} \label{tau-ramified-Galois}
	$ K(\pi_{ D})/ K^{\langle\tau\rangle}$ is a Galois extension with Galois group $\langle\tilde\tau\rangle\times\Gal( K(\pi_ D)/ K)$.
\end{lemma}
\begin{proof}
	On the one hand, the extension $\tilde\tau$ defined in \cref{extending-tau-uniquely} generates a subgroup $\langle\tilde\tau\rangle\subseteq\Aut_{ K^{\langle\tau\rangle}}\left( K(\pi_{ D})\right)$ of order $d$. The subgroup $\Gal( K(\pi_{ D})/ K)\subseteq\Aut_{ K^{\langle\tau\rangle}}\left( K(\pi_{ D})\right)$
	has order $s$, which is coprime to $d$. Consequently, the intersection of these two subgroups is trivial. 
	An automorphism in $\Gal( K(\pi_ D)/ K)$ maps $\pi_ D\mapsto \zeta_s^i\pi_ D$ for some $i$, hence $\tilde\tau$ commutes with all such automorphisms.
	We conclude that there is an embedding
	\[\langle\tilde\tau\rangle\times\Gal\left( K(\pi_{ D})/ K\right)\hookrightarrow \Aut_{ K^{\langle\tau\rangle}}\left( K(\pi_{ D})\right).\]
	The group on the right has order at most $\left( K(\pi_{ D}):  K^{\langle\tau\rangle}\right) = \#\langle\tilde\tau\rangle\cdot\#\Gal\left( K(\pi_{ D})/ K\right)$, so the embedding above is surjective, and the claim follows.
\end{proof}

\begin{proposition} \label{extending-tau-to-D}
	Let $ K/ k$ be a Galois $p$-extension of local fields, and let $\tau\in\Gal( K/ k)$ have order $d$. Let $ D$ be a skew field with centre $ K$ and Schur index $s$, and assume that $s\mid( q_\tau-1)$.
	Then there is a unique extension of $\tau$ to an element $\uptau\in\Aut_{ k} ( D)$ of order $d$.
\end{proposition}

\begin{proof}
	The field extension $ K(\omega)/ K$ is unramified of degree $s$, and $s$ is coprime to $d$, there is an isomorphism
	\begin{align*}
		\Gal\left( K(\omega) /  K^{\langle\tau\rangle}\right) &\simeq \Gal\left( K(\omega)/ K\right) \times \Gal\left( K/ K^{\langle\tau\rangle}\right)  .
	\end{align*}
	We let $\hat\tau$ be the Galois automorphism corresponding to the pair $(\id,\tau)$. Note that this does \emph{not} mean that $\hat\tau$ acts trivially on $\omega$.
	
	Define an extension $\uptau$ of $\tau$ to $ D$ by setting $\uptau(\omega)\colonequals \hat\tau(\omega)$ and $\uptau(\pi_{ D})\colonequals \tilde\tau(\pi_{ D}) = \epsilon_ D \pi_{ D}$. It is easily seen that this defines a homomorphism, and it is clear from \cref{extending-tau-uniquely} that $\uptau$ has the prescribed order and that it is unique.
\end{proof}
{In the sequel, we simply write $\tau$ for $\uptau$.}

\section{Skew power series rings} \label{sec:skew-power-series}
We retain the notation of \cref{sec:extending-Galois-Deta}. For simplicity, we will assume that $k=K^{\langle\tau\rangle}$. In this section, we study the skew power series ring $\OO_ D[[X;\tau,\delta]]$ where $\delta\colonequals\tau-\id$. Formal skew power series rings {have been studied in general by Schneider and Venjakob \cite{VenjakobWPT,SchneiderVenjakob}.}

\subsection{Well-definedness and basic lemmata} \label{sec:props-skew}
Let $\OO_{ D}$ denote the unique maximal $\OO_{ K}$-order in $ D$, {and let $\delta\colonequals \tau-\id$. The skew power series ring $\OO_{ D}[[X;\tau,\delta]]$ has the same underlying additive group as the ring of formal power series $\OO_D[[X]]$, with multiplication defined by the formula $Xd=\tau(d) X+\delta(d)$ for all $d\in\OO_D$.}

{To see that this gives rise to a well-defined ring structure, we note the following.
The ring $\OO_D$ is a noetherian pseudocompact ring in the sense of \cite[351]{SchneiderVenjakob}, and $\tau$ is a topological automorphism of it by \cref{extending-tau-non-uniquely}. The map $\delta:\OO_D\to\OO_D$ is continuous, and it commutes with $\tau$, that is, $\tau\circ\delta=\delta\circ\tau$. Moreover, $\delta$ is a left $\tau$-derivation, that is,
\[\forall a,b\in \OO_D:\quad \delta(ab)=\delta(a)b+\tau(a)\delta(b),\]
and it is $\tau$-nilpotent in the sense of \cite[354]{SchneiderVenjakob}.
For $\tau$ and $\delta$ possessing these properties, well-definedness of skew power series rings in general is shown in \cite[§§0--1]{SchneiderVenjakob}. An explicit proof of well-definedness of $\OO_{ D}[[X;\tau,\delta]]$ is given in \cite[\S3.3.1]{thesis}.}

We record the following identities for later use. In both cases, the proof is by induction and using the definition of $\delta$.
\begin{lemma} For $n\ge 0$ and $d\in\OO_{ D}$:
\begin{lemmalist}
    \item \label{delta-n} $\delta^n(d)=\sum_{\ell=0}^n (-1)^{n-\ell}\binom{n}{\ell}\tau^\ell(d)$,
    \item \label{Xnd} $X^n d = \sum_{i=0}^n \binom{n}{i} \tau^i \delta^{n-i}(d) X^i $. \qed
\end{lemmalist}
\end{lemma}

\subsection{The centre}
We will show that the centre of the skew power series ring $\OO_{ D}[[X;\tau,\tau-\id]]$ is the power series ring $\OO_{ k}[[(1+X)^{( K: k)}-1]]$. One containment can be seen through a straightforward albeit nontrivial computation, whereas the converse requires embedding the skew power series ring into a matrix ring, using ideas going back to Hasse \cite[Satz~40]{Hasse}.

\begin{proposition} \label{centre-of-skew-power-series-ring}
	The centre of $\OO_{ D}[[X;\tau,\tau-\id]]$ contains $\OO_{ k}[[(1+X)^{( K: k)}-1]]$.
\end{proposition}

\begin{proof}[Proof of \cref{centre-of-skew-power-series-ring}]
	The automorphism $\tau$ fixes $ k$, therefore every element of $\OO_{ k}$ is central. It remains to show that for all $d\in\OO_{ D}$,
	\[(1+X)^{( K: k)} d = d (1+X)^{( K: k)}\]
	We expand both sides using the binomial theorem, and show that the respective coefficients agree. The right hand side is
	\begin{equation}
		d (1+X)^{( K: k)} = \sum_{i=0}^{( K: k)} d \binom{( K: k)}{i} X^i  . \label{Xd-RHS}
	\end{equation}
	We compute the left hand side. First, by \cref{Xnd}, we get
	\begin{align}
		(1+X)^{( K: k)} d &= \sum_{j=0}^{( K: k)} \binom{( K: k)}{j} X^j d = \sum_{i=0}^{( K: k)} \sum_{j=i}^{( K: k)} \binom{( K: k)}{j} \binom{j}{i} \tau^i \delta^{j-i}(d) X^i  . \label{Xd-LHS}
	\end{align}
	
	The coefficient of $X^i$ on the right hand side of \eqref{Xd-LHS} is, by \cref{delta-n},
	\begin{align*}
		\sum_{j=i}^{( K: k)} \binom{( K: k)}{j} \binom{j}{i} \tau^i\delta^{j-i}(d) = \sum_{\ell=0}^{( K: k)-i} \left(\sum_{j=i+\ell}^{( K: k)} (-1)^{j-i-\ell} \binom{( K: k)}{j} \binom{j}{i} \binom{j-i}{\ell}\right) \tau^{i+\ell}(d)  ,
	\end{align*}
	We claim that the coefficient in the brackets is
	\begin{equation} \label{eq:j-i-ell-coeffs}
		\sum_{j=i+\ell}^{( K: k)} (-1)^{j-i-\ell} \binom{( K: k)}{j} \binom{j}{i} \binom{j-i}{\ell}=\begin{cases}
			\binom{( K: k)}{i} & \text{if $\ell=( K: k)-i$} , \\
			0 & \text{otherwise} .
		\end{cases}
	\end{equation}
	For this, recall the following two identities for binomial coefficients, see e.g. \cite[§I.2.6, (20)]{Knuth} and \cite[§I.2.6, (23)]{Knuth}, respectively:
	\begin{align*}
	    \binom{( K: k)}{j} \binom{j}{i} &= \binom{( K: k)}{i} \binom{( K: k)-i}{j-i} & 0\le i,j \le (K:k),\\
        \sum_{m\in\ZZ} (-1)^{r-m} \binom{r}{m}\binom{m+t}{n} &= \binom{t}{n-r} & \forall r, n\in \ZZ, \, r\ge0, \forall t\in\RR.
	\end{align*}
	Using the first identity, and then applying the second one with $t\colonequals0$, the left hand side of \eqref{eq:j-i-ell-coeffs} becomes
	\begin{align*}
		\sum_{j=i+\ell}^{( K: k)} (-1)^{j-i-\ell} \binom{( K: k)}{j} \binom{j}{i} \binom{j-i}{\ell} &= \binom{( K: k)}{i} \binom{0}{\ell-(( K: k)-i)}  .
	\end{align*}
	The second binomial coefficient here is $1$ if $\ell=( K: k)-i$ and zero otherwise, which proves \eqref{eq:j-i-ell-coeffs}.
	
	Therefore the coefficient of $X^i$ on the right hand side of \eqref{Xd-LHS} is
	\[\binom{( K: k)}{i} \tau^{( K: k)}(d)=\binom{( K: k)}{i} d  , \]
	using that $\tau$ has order $( K: k)$. This agrees with the coefficient of $X^i$ in \eqref{Xd-RHS} for all $i$.
\end{proof}

\begin{corollary} \label{full-centre-of-skew-power-series-ring}
	The centre of $\OO_{ D}[[X;\tau,\tau-\id]]$ is $\OO_{ k}[[(1+X)^{( K: k)}-1]]$. Moreover, the skew field $\Quot(\OO_{ D}[[X;\tau,\tau-\id]])$ has Schur index $( K: k) s$.
\end{corollary}
We introduce the following notation. For brevity, we will write $T\colonequals (1+X)^{( K: k)}-1$.
\begin{align*}
	\Sigmachi &\colonequals\OO_{ D}[[X;\tau,\tau-\id]]  , &
	\Dchi &\colonequals\Quot(\Sigmachi)  , \\
	\OO_\Emax &\colonequals \OO_{ K(\omega)}[[(1+X)^{( K: k)}-1]] = \OO_{ K(\omega)}[[T]]  , &
	\Emax &\colonequals\Frac \left(\OO_\Emax\right)  , \\
	\OO_\LLL&\colonequals \OO_{ K}[[(1+X)^{( K: k)}-1]] = \OO_{ K}[[T]]  , &\LLL&\colonequals \Frac\left(\OO_\LLL\right)  , \\
	\OO_\FFF &\colonequals \OO_{ k}[[(1+X)^{( K: k)}-1]] = \OO_{ k}[[T]]  , &
	\FFF &\colonequals \Frac\left(\OO_\FFF\right)  .
\end{align*}
{The automorphism $\tau$ of $D$ extends to an automorphism of $\Dchi$ by acting on coefficients. We will also denote this extension by $\tau$.}

\begin{proof}
	First, observe that $\Sigmachi$ is a (noncommutative) domain, because $\OO_{ D}$ is a domain and $\tau$ is an automorphism: this is \cite[Corollary~2.10(i)]{VenjakobWPT}. Therefore the total ring of quotients $\Dchi$ is a skew field. \Cref{centre-of-skew-power-series-ring} then shows that the square of the Schur index of this skew field divides the dimension over $\FFF$. This dimension is $\dim_{\FFF} \Dchi = ( K: k)( D: k)= \left( K: k\right)^2 s^2$. On the other hand, $\Dchi$ is a left $\Emax$-vector space of dimension $\dim_{\Emax} \Dchi = ( K: k)( D: K(\omega))= \left( K: k\right) s$.
	Note that $\Emax$ is a subfield of $\Dchi$. We will show that it is in fact a maximal (self-centralising) subfield, which in turn implies the \namecref{full-centre-of-skew-power-series-ring} by comparing the dimensions, {using \cite[Theorem~7.15]{MO}.}
	
	Consider the following composite map:
	\begin{align}
		\phi:  D &\to  K(\omega) \otimes_{ K}  D \xrightarrow{\sim} M_{s}( K(\omega)) \label{eq:phi-def} \\
		d &\mapsto 1\otimes d \notag
	\end{align}
	The second map here is the splitting isomorphism of Hasse, see \cite[Theorem~14.6]{MO}:
    \begin{align}
    	 K(\omega)\otimes_ K D&\xrightarrow{\sim}M_s( K(\omega)) \label{eq:Reiner-splitting} \\
    	x\otimes 1&\mapsto x\cdot\mathbf 1_s & {\text{for $x\in K(\omega)$}} \notag \\
    	1\otimes x&\mapsto \diag\left(x, \sigma(x),\ldots,\sigma^{s-1}(x)\right) & {\text{for $x\in K(\omega)$}} \notag \\
    	1\otimes \pi_ D &\mapsto \begin{psmallmatrix}
    		0&1 \\ &&1 \\ \vdots&&&\ddots \\ &&&&1 \\ \pi &&& \dots & 0 \\
    	\end{psmallmatrix} \notag
    \end{align}
    This map is an isomorphism onto its image. Moreover, $ D$ is a vector space over $ K(\omega)$ with basis $\pi_{ D}^i$ for $0\le i<s$, and the image $\phi( D)$ has $ K(\omega)$-basis $\phi(\pi_{ D}^i)$ for $0\le i<s$.
	
	The map $\phi$ gives rise to a homomorphism 
	\begin{align*}
		\hat\phi: \Quot\left(\OO_{ D}[[T]]\right) \to \Emax\otimes_\LLL \Quot\left(\OO_{ D}[[T]]\right) \xrightarrow\sim M_{s}\left(\Emax\right)
	\end{align*}
	by letting $\hat\phi(T)=T \cdot \mathbf 1_s$.	
	We define an $\LLL$-vector space homomorphism
	\begin{equation}\label{eq:Phi-def}
		\Phi: \Dchi\to \Emax \otimes_{\LLL} \Dchi \to M_{( K: k) s}\left(\Emax\right)  .
	\end{equation}
	The first map sends an element $d\in\Dchi$ to $1\otimes d$. The second map is the unique $\Emax$-linear map determined by the following two properties:
	for $g\in\Quot(\OO_{ D}[[T]])$, it maps $1\otimes g$ to the block diagonal matrix
	\[\Phi(g)\colonequals \diag\left(\hat\phi(g),\tau(\hat\phi(g)), \tau^2(\hat\phi(g)),\ldots,\tau^{( K: k)-1}(\hat\phi(g))\right)  ,\]
	and it maps $1\otimes X$ to
	\[\Phi(X)\colonequals \begin{psmallmatrix}
		-\mathbf{1}_{s} & \mathbf{1}_{s} \\
		& -\mathbf{1}_{s} & \mathbf{1}_{s} \\
		&& \ddots & \ddots \\
		&&& \ddots & \mathbf{1}_{s} \\
		&&&& -\mathbf{1}_{s} & \mathbf{1}_{s} \\
		(1+X)^{( K: k)}\mathbf{1}_{s} &&&&& -\mathbf{1}_{s}
	\end{psmallmatrix} = \begin{psmallmatrix}
		& \mathbf{1}_{s} \\
		& & \mathbf{1}_{s} \\
		&&& \ddots \\
		&&&& \mathbf{1}_{s} \\
		&&&&& \mathbf{1}_{s} \\
		(1+X)^{( K: k)}\mathbf{1}_{s} &&&&&
	\end{psmallmatrix}-\mathbf{1}_{( K: k)s}  ,\]
	where $\mathbf{1}_{s}$ is the $s\times s$ identity matrix.
	This defines a homomorphism as in \eqref{eq:Phi-def}, and for all $d\in  D$, the multiplication rule $\Phi(1+X)\Phi(d)=\tau(\Phi(d))\Phi(1+X)$
	of skew power series rings is satisfied: {see \cite[63]{thesis} for details.}
	So $\Phi$ is an $\LLL$-algebra homomorphism. Since $\Dchi$ is a skew field and $\Phi$ is not the zero map, $\Phi$ is injective.
	
	Since $\Phi$ is an isomorphism onto its image, $\Emax$ is a maximal subfield of $\Dchi$ if and only if $\Phi(\Emax)$ is a maximal subfield in $\Phi(\Dchi)$. 
	Let $\alpha$ be a primitive element of the field extension $ K/ k$. Then the Galois conjugates $\tau^j(\alpha)$ are pairwise distinct for $0\le j<( K: k)$. The elements $\sigma^i(\omega)$ are also pairwise distinct for $0\le i<s$. Since $\tau$ and $\sigma$ have coprime orders, the diagonal matrix $\Phi(\alpha\omega)$ has pairwise distinct entries in its diagonal. 
	
	Suppose $A\in\Phi(\Dchi)$ centralises $\Phi(\Emax)$. Then in particular, it commutes with $\Phi(\alpha\omega)$. By the previous observation, this forces $A$ to be diagonal. But it is clear from the definition of $\Phi$ that diagonal matrices have preimage in $\Emax$, hence $A\in\Phi(\Emax)$, proving that $\Emax$ is a maximal (self-centralising) subfield of $\Dchi$.
\end{proof}

\subsection{Maximal orders and maximal subfields} \label{sec:further-props-skew}
We collect a few more properties of the skew power series ring $\OO_{ D}[[X;\tau,\delta]]$. Let us write $\cent(\Dchi)$ and $\OO_{\cent(\Dchi)}$ for $\FFF$ and $\OO_\FFF$ from now on.

\begin{lemma} \label{Sigmachi-maximal}
$\Sigmachi$ is a maximal $\OO_{\cent(\Dchi)}$-order in $\Dchi$.
\end{lemma}
\begin{proof}
The statement follows from \cite[Corollary 2.10(iii)]{VenjakobWPT}, which states that if $ R$ is a complete local ring with maximal ideal $\mm$, $ A= R[[T;\sigma,\delta]]$ a ring of skew power series over $ R$, the {associated graded ring $\gr_\mm  R$ of $R$} is a noetherian maximal order in $\Quot(\gr_\mm  R)$, and the reduction $\overline\sigma$ of $\sigma$ to $ R/\mm$ is an automorphism, then $ A$ is a noetherian maximal order in $\Quot( A)$.

We check that the conditions of this Corollary are satisfied. In Venjakob's notation, $ R\colonequals\OO_{ D}$, $\mm\colonequals\pi_{ D}\OO_{ D}$, $ A\colonequals\Sigmachi$. 
The associated graded ring $\gr_{\pi_{ D}\OO_{ D}} \OO_{ D}$ is isomorphic to the polynomial ring $\overline{\OO_{ D}}[t]$ where $\overline{\OO_{ D}}=\OO_{ D}/\pi_{ D}\OO_{ D}$ is the residue field; a priori, this {is} a skew field, but seeing that it is the same as $\overline{ K}(\omega)$ shows that it is actually a field. Just as in the commutative case, the isomorphism is given by
\begin{align*}
	\gr_{\pi_{ D}\OO_{ D}} \OO_{ D} = \bigoplus_{i=0}^\infty \OO_{ D}\pi_{ D}^i/\OO_{ D}\pi_{ D}^{i+1} &\xrightarrow{\sim} \overline{ K}(\omega)[t] \\
	\left[\pi_{ D}^i \bmod \pi_{ D}^{i+1}\right] &\mapsto t^i
\end{align*} 
The polynomial ring $\overline{ K}(\omega)[t]$ is {factorial}, hence normal, and thus a maximal order in its field of fractions $\overline{ K}(\omega)(t)$.
The remaining conditions are easily verified.
\end{proof}

{
\begin{lemma} \label{Sigmachi-conjugate}
    All maximal $\OO_{\cent(\Dchi)}$-orders in $\Dchi$ are conjugates of $\Sigmachi$.
\end{lemma}
\begin{proof}
    The ring $\OO_{\cent(\Dchi)}=\OO_{k}[[T]]$ is a regular local ring of dimension two, so Ramras's description of {quasi-local} orders {of finite global dimension} over regular local rings of dimension two {\cite[Theorem~5.4]{Ramras2}} is applicable.
    
    {The ring $\Sigmachi$ is local: its maximal ideal consists of skew power series with non-unit constant term \cite[Proposition~2.11]{VenjakobWPT}. The residue (skew) field is therefore $\OO_D/\pi_D\OO_D$: this is a finite field, which is in particular simple artinian, and hence $\Sigmachi$ is quasi-local in the sense of Ramras. The proof of \cref{Sigmachi-maximal} combined with \cite[Corollary~2.10(ii)]{VenjakobWPT} shows that $\Sigmachi$ has finite global dimension. Therefore $\Sigmachi$ satisfies the hypotheses of Ramras's theorem.}
    
    {Ramras's result states} that any two maximal orders in a central simple algebra over the field of fractions of a regular local ring of dimension two are conjugate. That is, for every maximal $\OO_{\cent(\Dchi)}$-order $\mathfrak M$ in $\Dchi$, there is some $u\in\Dchi^\times$ such that $\mathfrak M=u \Sigmachi u^{-1}$.
\end{proof}
}
\begin{remark}
    {The extent to which maximal orders in $\Dchi$ are not unique can be measured locally as follows. Let $\pp\subset \OO_{\cent(\Dchi)}$ be a prime ideal of height $1$. Let $\widehat{\OO}_{\cent(\Dchi),\pp}$ be the ring obtained by first localising $\OO_{\cent(\Dchi)}$ at $\pp$ and then taking completion in the $\pp$-adic topology. The field of fractions $\widehat{\cent(\Dchi)}_\pp\colonequals\Frac(\widehat{\OO}_{\cent(\Dchi),\pp})$ is a two-dimensional local field, that is, a discretely valued field whose residue field is a usual (i.e. one-dimensional) local field.} 
    
    {Define $\Dcp\colonequals \Frac(\widehat{\OO}_{\cent(\Dchi),\pp})\otimes_{\cent(\Dchi)}\Dchi$: then $\cent(\Dcp)=\widehat{\cent(\Dchi)}_\pp$, and $\Dcp$ is a central simple $\widehat{\cent(\Dchi)}_\pp$-algebra. If $\Dcp$ is a skew field, then it contains a unique maximal order by \cite[Theorem~12.8]{MO}. If $\pp$ is generated by a polynomial that admits a zero in some field extension of $k$ inside $D$, then it follows from a dimension counting argument that $\Dcp$ is a matrix ring over some skew field, and this skew field is finite dimensional over $\widehat{\cent(\Dchi)}_\pp$, and \cite[Theorem~17.3(ii)]{MO} shows that all maximal orders in $\Dcp$ are conjugate. From this point of view, \cref{Sigmachi-conjugate} can be seen as a local--global principle type statement.}

    {This approach of studying $\Dchi$ using the theory of higher local fields goes back to Lau \cite{Lau}. Generalising her methods, one can show that $\cent(\Dchi)$ has cohomological dimension $3$, and that $\SK_1(\Dcp)=1$ when $\pp\nmid p$. For further results as well as for a justification of the splitting behaviour of $\Dcp$ discussed above, we refer to Chapter~5 of the thesis \cite{thesis}.}
\end{remark}

{
\begin{remark}
    Lastly, we remark that $\Sigmachi$ is a UFD by \cite[Corollary~7.4]{VenjakobWPT}.
\end{remark}
}

\section{The Wedderburn decomposition} \label{sec:Wedderburn}

\subsection{Wedderburn decompositions and their skew fields} \label{sec:2-1}
The group ring $ F[H]$ is semisimple by Maschke's theorem. Consider its Wedderburn decomposition
\begin{equation} \label{eq:FH-Wedderburn}
    F[H]=\bigoplus_{\eta\in\Irr(H)/\sim_{ F}} M_{\nFeta}(D_\eta)  .
\end{equation}
Here $D_\eta$ is a finite dimensional skew field over its centre $ F(\eta)$, $\Irr( H)$ is the set of $ F^\al$-valued irreducible characters of $ H$, and two characters $\eta,\eta'\in\Irr( H)$ are equivalent, denoted $\eta\sim_{ F}\eta'$, if there exists some $\sigma\in\Gal( F(\eta)/ F)$ such that ${}^\sigma\eta=\eta'$ where $ F(\eta)= F(\eta(h):h\in H)$.
Let $\OO_{D_\eta}$ denote the unique maximal $\OO_ F$-order in $D_\eta$. 

The skew fields $D_\eta$ are cyclic algebras as in \cref{sec:extending-Galois-Deta}. Fix a uniformiser $\pi_\eta$ of $ F(\eta)$, and write $\pi_{D_\eta}$ for the element denoted by $\pi_{ D}$ in \cref{sec:extending-Galois-Deta}. Then $\pi_\eta=\pi_{D_\eta}^{\sFeta}$, and $\omega$ is a root of unity of order {$q_\eta^{\sFeta}-1$}, where $q_\eta$ is the order of the residue field of $ F(\eta)$, and $\sFeta$ denotes the Schur index of $\eta$. 
{The root of unity $\omega$ depends on $\eta$; as we will mostly work in a single Wedderburn component, we shall suppress this dependence in our notation.}
Note that $s_\eta\mid p-1$ by \cite[Satz 10]{Witt1952}. Recall the well-known relationship $\eta(1)=\sFeta \nFeta$
between degree, Schur index and the size of the matrix ring, see e.g. \cite[Remark~74.10(ii)]{CRII}. Note that $\eta(1)$ is independent of $ F$ while $\sFeta$ and $\nFeta$ are not.

Setting $\tD_\eta\colonequals \Q^{ F(\eta)}(\Gamma_0) \otimes_{ F(\eta)} D_\eta$, the Wedderburn decomposition of $ F[H]$ induces the following decomposition of $\Q^ F(\Gamma_0)[H]$:\begingroup\allowdisplaybreaks
\begin{align} \label{eq:QFGamma0-Wedderburn}
	\Q^ F(\Gamma_0)[H] &= \bigoplus_{\eta\in\Irr(H)/\sim_{ F}} M_{\nFeta} \big( \tD_\eta \big)  .
\end{align}
\begin{lemma}
	$\tD_\eta$ is a skew field with centre $\cent(\tD_\eta)=\Q^{ F(\eta)}(\Gamma_0)$ and Schur index $\sFeta$.
\end{lemma}
\begin{proof}
	The ring $\tD_\eta$ is a simple algebra with centre $\cent(\tD_\eta)=\Q^{ F(\eta)}(\Gamma_0)$ by \cite[Theorem~7.6]{MO}, and it is a skew field by essentially the same argument as in \cite[Lemma~2.10]{NichiforPalvannan}.
    Since $ F(\eta)(\omega)$ is a splitting field for $D_\eta$, the subfield $\Q^{ F(\eta)(\omega)}(\Gamma_0)$ of $\tD_\eta$ is also a splitting field, and $\tD_\eta\otimes_{\Q^{ F(\eta)}(\Gamma_0)}\Q^{ F(\eta)(\omega)}(\Gamma_0) \simeq M_{\sFeta}\left(\Q^{ F(\eta)(\omega)}(\Gamma_0)\right)$, hence $\tD_\eta$ has Schur index $s_\eta$.
\end{proof}

The algebra $\Q^ F(\G)$ is semisimple artinian {\cite[Proposition~5(1)]{TEIT-II}}. We write
\[\Q^ F(\G) \simeq \bigoplus_{\chi\in\Irr(\G)/\sim_{ F}} M_{\nFchi}(\DFchi)\]
for its Wedderburn decomposition. {Recall that $\chi\sim_F\chi'$ if there is a $\sigma\in\Gal( F_\chi/ F)$ such that ${}^\sigma(\res^\G_H \chi)=\res^\G_H \chi'$.}
Let $\sFchi$ denote the Schur index of $\DFchi$: then we have $\chi(1)=\sFchi \nFchi$, see \cite[Corollary~1.9]{NickelConductor}.

\begin{definition} \label{def:f-idempotent}
	For $1\le j\le \nFeta$, let $f_\eta^{(j)}$ be the element corresponding to the $\nFeta\times \nFeta$ matrix with zeros everywhere but in the $j$th entry of the diagonal and $1$ there, under the isomorphism $\epsilon(\eta) F[H] \simeq M_{\nFeta}(D_\eta) $.
\end{definition}
These are indecomposable, not necessarily central idempotents, and the following identities hold:
\begin{gather}
\big(f_\eta^{(j)}\big)^2 = f_\eta^{(j)} = f_\eta^{(j)} \epsilon(\eta), \quad \epsilon(\eta)=\sum_{j=1}^{\nFeta} f_\eta^{(j)} , \notag \\
\label{Deta-feta}
	D_\eta\simeq \begin{psmallmatrix}
		0 \\
		& \ddots \\
		&  & 1 \\
		&  &  & \ddots \\
		&  &  &  & 0 
	\end{psmallmatrix} M_{\nFeta}(D_\eta) \begin{psmallmatrix}
		0 \\
		& \ddots \\
		&  & 1 \\
		&  &  & \ddots \\
		&  &  &  & 0 
	\end{psmallmatrix} \simeq f_\eta^{(j)}  F[H]\epsilon(\eta) f_\eta^{(j)} = f_\eta^{(j)}  F[H] f_\eta^{(j)}  .
\end{gather}

The above constructions in the group ring $F[H]$ also work in the group ring $\Q^ F(\Gamma_0)[H]$, that is, after tensoring with $\Q^ F(\Gamma_0)$: 
\begin{equation} \label{tDeta-feta}
	\tD_\eta\simeq f_\eta^{(j)} \Q^ F(\Gamma_0)[H] f_\eta^{(j)}  .
\end{equation}
Here we tacitly identify $f_\eta^{(j)}$ with $1\otimes f_\eta^{(j)}$, and we shall continue to do so in the sequel.
These idempotents were used to study Wedderburn decompositions of Iwasawa algebras in \cite[Theorem 1.11]{NickelConductor} and in \cite[Theorem~1]{Lau}.

\label{def:fchi}Analogously, we can define $f_\chi^{(j)}$ to be the element corresponding to the $n_\chi\times n_\chi$-matrix with $1$ in the $j$th diagonal position and zeros elsewhere under the isomorphism $\epsilon_\chi\Q^ F(\G) \simeq M_{\nFchi}(\DFchi)$.
The relationship between the idempotents $f_\chi^{(j)}$ resp. $f_\eta^{(j')}$ associated with $\chi$ resp. $\eta$, where $\eta$ is an irreducible constituent of $\res_H^\G\chi$, will be studied in \cref{sec:dimSchur}.

\subsection{The extension \texorpdfstring{$ F(\eta)/ F_\chi$}{F(η)/Fχ}} \label{sec:Feta-Fchi}

As before, let $ F$ be a finite extension of $\QQ_p$.

\begin{definition}
	Let $\vf$ be the minimal positive exponent such that $\gamma^{\vf}$ acts as a Galois automorphism on $\eta$:
	\[\vf\colonequals \min\left\{0< i\le w_\chi : \exists \tau\in\Gal( F(\eta)/ F), {}^{\gamma^i}\eta={}^\tau\eta\right\}  .\] 
\end{definition}
{As we will momentarily see in \cref{Qpeta-Qpchi-cyclic}, $v_\chi$ does not depend on $\eta$.}
Note that $\vf\mid w_\chi$ since ${}^{\gamma^{w_\chi}}\eta=\eta$ by definition of $w_\chi$.
\begin{remark}
	The number $v^{\QQ_p}_\chi$ was introduced in \cite[1223]{Lau}. There is a typo in that definition: if $i=0$ would be allowed, then one would have $v^{\QQ_p}_\chi=0$ with $\tau=\id$.
\end{remark}

Recall from \eqref{eta-i} that $\sss$ is the number of irreducible characters of $H$ whose Galois orbits sum up to $\res^\G_H\chi$.
\begin{proposition}\label{Qpeta-Qpchi-cyclic}
	The Galois group $\Gal( F(\eta)/ F_{\chi})$ is cyclic of order $w_\chi/\vf$, and $\sss=\vf$.
	Any $\tau\in\Gal( F(\eta)/ F)$ such that ${}^{\gamma^{\vf}}\eta={}^\tau\eta$ is a generator of $\Gal( F(\eta)/ F_{\chi})$, and in fact, there is exactly one such $\tau$. 
\end{proposition}
In the sequel, the symbol $\tauf=\tau^ F$ shall denote this unique automorphism (or an extension of it, as constructed in \cref{sec:extending-Galois-Deta}). {Dependence of $\tau$ on $\chi$ is suppressed from the notation.} We shall no longer use the non-standard notation $\sss$, always writing $\vf$ instead. {It follows from \cref{Qpeta-Qpchi-cyclic} and \eqref{wchi-s} that $v_\chi$ depends only on $\chi$ and not on $\eta$.}

\begin{remark}
    \Cref{Qpeta-Qpchi-cyclic} is an improvement upon \cite[Lemma~1.1]{NickelConductor}, where it was shown that $w_\chi=\sss( F(\eta): F_\chi)$. It also shows that $\vf$ does not depend on the choice of the topological generator $\gamma$. Furthermore, it is a direct consequence of \cref{Qpeta-Qpchi-cyclic} that the field $L$ occurring in \cite{Lau} is in fact $\QQ_{p,\chi}$; see Proposition~1 and Theorem~1 of loc.cit.
\end{remark}

\begin{proof}[Proof of \cref{Qpeta-Qpchi-cyclic}]
	Let $\tau\in\Gal( F(\eta)/ F)$ be as in the definition of $\vf$. First, note that
	\begin{align*}
		\res^\G_H\chi = \sum_{i=0}^{w_\chi-1} {}^{\gamma^i}\eta = \sum_{i=0}^{w_\chi-1} {}^{\gamma^i}\left({}^{\gamma^{\vf}}\eta\right) 
		= \sum_{i=0}^{w_\chi-1} {}^{\gamma^i}\left({}^{\tau}\eta\right)
		= \tau\left(\sum_{i=0}^{w_\chi-1} {}^{\gamma^i}\eta\right)
		= \tau\left(\res^\G_H\chi\right),
	\end{align*}
    which shows that $\tau\in\Gal( F(\eta)/ F_{\chi})$.
	Consider the characters ${}^{\gamma^k}\eta$ for $0\le k <\vf$. These are all in separate $\Gal( F(\eta)/ F)$-orbits: that is, for any $\psi\in\Gal( F(\eta)/ F)$ and $0\le k,k'< \vf$ distinct, we have 
	\begin{equation}\label{eq:gamma-separate-orbits}
		{}^{\gamma^k}\eta\ne{}^\psi\left({}^{\gamma^{k'}}\eta\right)  .
	\end{equation}
	by the minimality condition in the definition of $\vf$.
	
	Now consider the following two decompositions of $\res_H^\G\chi$:
	\begin{equation} \label{eq:sum-eta-decompositions}
		\sum_{j=0}^{\sss-1} \sum_{\psi\in\Gal( F(\eta)/ F_{\chi})} {}^\psi \eta_j = \res_H^\G \chi = \sum_{i=0}^{w_\chi-1} {}^{\gamma^i}\eta = \sum_{k=0}^{\vf-1}\sum_{\ell=0}^{\frac{w_\chi}{\vf}-1} {}^{\tau^\ell}\left({}^{\gamma^k}\eta\right)  .
	\end{equation}
	where $\eta_j$ was defined in \eqref{eta-i}.
	The summands on the right hand side are distinct, that is, \[{}^{\tau^\ell}\left({}^{\gamma^k}\eta\right)\ne {}^{\tau^{\ell'}}\left({}^{\gamma^{k'}}\eta\right) \text{ unless } \ell=\ell' \text{ and } k=k'  .\]
	Indeed, for $k\ne k'$, this is \eqref{eq:gamma-separate-orbits}. For $k=k'$, this is because ${}^{\tau^\ell}({}^{\gamma^k}\eta)={}^{\tau^{\ell'}}({}^{\gamma^k}\eta)$ is equivalent to $\eta={}^{\tau^{\ell'-\ell}}\eta$, which means that $\tau^{\ell'-\ell}$ fixes $ F(\eta)$, and so $\ell'=\ell$. (In fact, the same argument shows that if $\psi,\psi'\in\Gal( F(\eta)/ F)$ then ${}^{\psi}({}^{\gamma^k}\eta)\ne {}^{\psi'}({}^{\gamma^{k'}}\eta)$ unless $\psi=\psi'$ and $k=k'$.)
	
	On the left hand side of \eqref{eq:sum-eta-decompositions}, we have full Galois orbits, so this must also be true on the right hand side. It follows that $\sss=\vf$, and we may assume $\eta_{(k)}={}^{\gamma^k}\eta$ up to renumbering. Moreover,  $\Gal( F(\eta)/ F_{\chi})$ is generated by $\tau$; in particular, it is cyclic of order $w_\chi/\vf$.
	
	Finally, we show that there is only one $\tau\in\Gal( F(\eta)/ F_{\chi})$ for which ${}^{\gamma^{\vf}}\eta={}^\tau\eta$. {Suppose that ${}^{\gamma^{\vf}}\eta={}^{\tau'}\eta$ for some $\tau'\in\Gal(F(\eta)/F)$. Then ${}^\tau\eta={}^{\tau'}\eta$, so $\tau(\eta(h))=\tau'(\eta(h))$ for all $h\in H$. Since $F(\eta)$ is generated by the values $\eta(h)$ over $F$, this shows $\tau=\tau'$.}
\end{proof}

\begin{lemma} \label{changing-v}
	Let $ F_\chi\subseteq E\subseteq  F(\eta)$, and let $\tau^ F$ resp. $\tau^E$ denote the automorphisms given by \cref{Qpeta-Qpchi-cyclic} for the ground fields $ F$ resp. $E$. Then $\tau^E=(\tau^ F)^{(E: F_\chi)}$ and $v^E_\chi=v^ F_\chi (E: F_\chi)$.
\end{lemma}
\begin{proof}
Since the extension $ F(\eta)/ F_\chi$ is cyclic, it is clear that $(\tau^ F)^{(E: F_\chi)}\in\Gal( F(\eta)/E)$, and this is a generator. Moreover, ${}^{(\tau^ F)^{(E: F_\chi)}}\eta={}^{\gamma^{v^ F_\chi (E: F_\chi)}}\eta$, hence $v^E_\chi\mid v^ F_\chi (E: F_\chi)$. Let us write $v^E_\chi t=v^ F_\chi (E: F_\chi)$.

Suppose that $\psi\in\Gal( F(\eta)/E)$ is such that ${}^\psi\eta={}^{\gamma^{v^E_\chi}}\eta$. Then $\psi$ is also a generator, and ${}^{\psi^t}\eta={}^{\gamma^{v^E_\chi t}}\eta={}^{\gamma^{v^ F_\chi (E: F_\chi)}}\eta={}^{(\tau^ F)^{(E: F_\chi)}}\eta$, which shows that $\psi^t=(\tau^ F)^{(E: F_\chi)}$. Since both $\psi$ and $\psi^t=(\tau^ F)^{(E: F_\chi)}$ are generators of the finite $p$-group $\Gal( F(\eta)/E)$, we have $p\nmid t$. On the other hand, $v^ F_\chi (E: F_\chi)$ is a $p$-power, and thus so is $t$. Hence $t=1$, which finishes the proof.
\end{proof}

The results of \cref{sec:extending-Galois-Deta} are applicable to $D_\eta$: indeed, the extension $ F(\eta)/ F_\chi$ is a Galois $p$-extension (its degree divides the $p$-power $w_\chi$), we have $ F(\eta)^{\langle\tauf\rangle}= F_\chi$ by \cref{Qpeta-Qpchi-cyclic}, and the condition on the Schur index is satisfied, as we have divisibilities $\sFeta\mid (p-1)\mid ( q_{\tauf}-1)$ by \cite[Satz 10]{Witt1952}.
So $\tau$ extends to $D_\eta$ as an automorphism of order $w_\chi/v_\chi$. Finally, we extend $\tauf$ from $D_\eta$ to $\tD_\eta$ by letting $\tauf(\gamma_0)\colonequals\gamma_0$.

\subsection{Description of the Wedderburn decomposition} \label{sec:theorems}
The following \cref{dim-Schur,D-conj} describe the Wedderburn decomposition 
\[\Q^ F(\G) \simeq \bigoplus_{\chi\in \Irr({\G})/\sim_ F} M_{\nFchi}(\DFchi)  .\]
{Here $F/\QQ_p$ is a finite extension, and two irreducible characters $\chi,\chi'$ with open kernel are equivalent if there is a $\sigma\in\Gal( F_\chi/ F)$ such that ${}^\sigma(\res^\G_H \chi)=\res^\G_H \chi'$. For each such equivalence class, $D_\chi$ is a skew field of index $s_\chi$.}

{
\begin{theorem} \label{dim-Schur}
{Let $\chi\in\Irr(\G)$ be an irreducible character with open kernel, and let $\eta\mid\res^\G_H\chi$ be an irreducible constituent of its restriction to $H$.}
{Suppose that $F(\eta)/F_\chi$ is totally ramified.}
Then
\begin{theoremlist}
	\item $\nFchi = \nFeta \vf$, \label{sn-inequality} \label{s-neta-nchi}
	\item $\sFchi = \sFeta w_\chi/\vf$. \label{s-chi-s-eta}
\end{theoremlist}
\end{theorem}
Recall that $n_\eta$ and $s_\eta$ are independent of the choice of $\eta$, see \eqref{eta-i}.
\begin{theorem} \label{D-conj}
{Let $\chi\in\Irr(\G)$ be an irreducible character with open kernel, and let $\eta\mid\res^\G_H\chi$ be an irreducible constituent of its restriction to $H$.}
{Suppose that $F(\eta)/F_\chi$ is totally ramified.}
Then
\[\DFchi \simeq \Quot\left(\OO_{D_\eta} [[ X, \tau, \tau-\id ]]\right)  .\]
\end{theorem}
}

We will prove \cref{dim-Schur,D-conj} in \cref{sec:TORC}.
\begin{remark} \label{rmk:totally-ramified-hypothesis}
    Two important cases in which all the extensions $F(\eta)/ F_{\chi}$ are totally ramified are the following:
    \begin{enumerate}
        \item $F$ is a finite extension of $\QQ_p$ of ramification degree prime to $p$, and $H$ is such that $p\nmid q-1$ holds for every prime factor $q\mid \# H$;
        \item $\G\simeq H\times \Gamma$ is a direct product (in fact, we have $F(\eta)=F_\chi$ in this case).
    \end{enumerate}
    In particular, the first condition holds for all pro-$p$-groups $\G$, so our results generalise those of \cite[\S2]{Lau}.
    The first assertion holds because $F(\eta)/ F_{\chi}$ is a subquotient of a base change of a cyclotomic extension, and the second assertion follows from \cite[Theorem~4.21]{Isaacs}.

    {It is expected that even in the case of arbitrary ramification, $D_\chi$ is isomorphic to the total ring of quotients of a skew power series ring as in \cref{sec:skew-power-series}.}
\end{remark}

In arithmetic applications, it is often desirable to have a description of all maximal orders in $\Q^F(\G)$, which we now provide; this result won't be needed in the rest of this work. 
{Suppose that for each $\chi\in\Irr(\G)$, the extension $F(\eta)/F_\chi$ is totally ramified.}
For each $\chi\in\Irr(\G)$, let {$\Sigmachi_\chi\colonequals \OO_{D_\eta} [[ X, \tau, \tau-\id ]]$} be the skew power series ring occurring in \cref{D-conj}, and let $\Dchi_\chi\colonequals\Quot(\Sigmachi_\chi)$ be its skew field of fractions.

\begin{corollary} \label{all-maximal-orders-in-QG}
{Suppose that for each $\chi\in\Irr(\G)$, the extension $F(\eta)/F_\chi$ is totally ramified.}
Then every maximal $\Lambda^{\OO_ F}(\Gamma_0)$-order in $\Q^ F(\G)$ is isomorphic to one of the form
\[\bigoplus_{\chi\in\Irr(\G)/\sim_{ F}} u_\chi M_{\nFchi} \left( \Sigmachi_\chi \right) u_\chi^{-1} \]
where $u_\chi\in \GL_{\nFchi}\left(\Dchi_\chi\right)$.
\end{corollary}
\begin{proof}
{\Cref{Sigmachi-conjugate} states that the maximal $\OO_{\cent(\Dchi_\chi)}$-orders in $\Dchi_\chi$ are of the form $\upsilon_\chi\Sigmachi_\chi \upsilon_\chi^{-1}$ for $\upsilon_\chi\in\Dchi_\chi^\times$. Therefore each $M_{\nFchi}(\upsilon_\chi\Sigmachi_\chi \upsilon_\chi^{-1})=\upsilon_\chi M_{\nFchi}(\Sigmachi_\chi) \upsilon_\chi^{-1}$ is a maximal $\OO_{\cent(\Dchi_\chi)}$-order in $M_{\nFchi}(\Dchi_\chi)$ by \cite[Theorem~8.7]{MO}, and therefore so is each $u_\chi M_{\nFchi}(\Sigmachi_\chi) u_\chi^{-1}$ for $u_\chi\in\GL_{n_\chi}(\Dchi_\chi)$.}

{Conversely, $\OO_{\cent(\Dchi_\chi)}=\OO_{F_\chi}[[T]]$ is a regular local ring of dimension two, so we can apply Ramras's result \cite[Theorem~5.4]{Ramras2} as in the proof of \cref{Sigmachi-conjugate}. 
{The quasi-local condition is satisfied by combining the argument there with \cite[Proposition~5.14]{CR}, and $M_{\nFchi}(\Sigmachi_\chi)$ has finite global dimension because $\Sigmachi_\chi$ does, since the global dimension is Morita invariant.}
We conclude that for every maximal $\OO_{\cent(\Dchi_\chi)}$-order $\mathfrak M_\chi$ in $M_{\nFchi}(\Dchi_\chi)$, there is some $u_\chi\in\GL_{n_\chi}(\Dchi_\chi)$ such that $\mathfrak M_\chi=u_\chi M_{\nFchi}(\Sigmachi_\chi) u_\chi^{-1}$.}

{The module $\OO_{\cent(\DFchi)}$ is finitely generated over $\Lambda^{\OO_ F}(\Gamma_0)$. It follows that $M_{\nFchi}(\Sigmachi_\chi)$ is also a maximal $\Lambda^{\OO_ F}(\Gamma_0)$-order in $M_{n_\chi}(\Dchi_\chi)$. The proof is concluded by invoking the fact that maximal orders behave well with respect to direct sums, cf.~\cite[Theorem~10.5(ii)]{MO}.}
\end{proof}

{Note that the proof of \cref{all-maximal-orders-in-QG} really concerns maximal orders in semisimple algebras whose skew field parts are total rings of quotients of skew power series rings as studied in \cref{sec:skew-power-series}.}

\section{Galois action and \texorpdfstring{$\Gamma$}{Γ}-action} \label{sec:actions}
Let $ F$ be a finite extension of $\QQ_p$. 
The Wedderburn decomposition of the group ring $ F[H]$ is well understood: the skew fields occurring can be described explicitly. Therefore the same holds for $\Q^ F(\Gamma_0)[H]$. A fundamental idea is that together with the decomposition \eqref{eq:Q-decomposition}, this provides a way to attack $\Q^ F(\G)$.

\subsection{\texorpdfstring{$\pc$}{δγ} and \texorpdfstring{$\pt$}{δτ}} \label{sec:deltas}

Consider the abstract Wedderburn isomorphism of the central simple $ F(\eta)$-algebras
\begin{equation} \label{eq:Wedderburn-eta-QQ}
	 F[H]\epsilon(\eta)\simeq M_{\nFeta}(D_\eta);
\end{equation}
{this is the $\eta$-component of \eqref{eq:FH-Wedderburn}, and $\epsilon(\eta)$ is the idempotent defined in \eqref{eq:def-epsilon}.}
Upon tensoring with $\Q^ F(\Gamma_0)$, this gives rise to the following isomorphism of central simple $\Q^{ F(\eta)}(\Gamma_0)$-algebras
\begin{equation} \label{eq:Wedderburn-eta}
	\Q^ F(\Gamma_0)[H]\epsilon(\eta) \simeq M_{\nFeta}(\tD_\eta),
\end{equation}
{where $\tD_\eta$ is the skew field in the $\eta$-component of the Wedderburn decomposition \eqref{eq:QFGamma0-Wedderburn} of $\Q^ F(\Gamma_0)[H]$.}
On the left hand side, we have an action of conjugation by $\gamma^{\vf}$.  This is because $H$ is a normal subgroup of $\G$, and  $\gamma^{\vf}e({}^\sigma\eta)\gamma^{-\vf}=e\left({}^{\tauf^{-1}\sigma}\eta\right)$ for all $\sigma\in\Gal( F(\eta)/ F)$, so conjugation by $\gamma^{\vf}$ acts by permuting the summands $e({}^\sigma\eta)$ of $\epsilon(\eta)$.

\begin{remark}
	Note that $\vf$ is the minimal positive integer for which conjugation by $\gamma^{\vf}$ acts on $ F[H]\epsilon(\eta)$. Indeed, in the proof of \cref{Qpeta-Qpchi-cyclic} we have seen that for $0\le k,k'<\vf$ and $\psi,\psi'\in\Gal( F(\eta)/ F)$, we have ${}^{\psi}({}^{\gamma^k}\eta)\ne {}^{\psi'}({}^{\gamma^{k'}}\eta)$ unless $\psi=\psi'$ and $k=k'$. Hence conjugation by $\gamma^k$ does not preserve $ F[H]\epsilon(\eta)$.
\end{remark}

{Recall that $\tauf$ is an automorphism of $D_\eta$. Setting $\tauf(\gamma_0)\colonequals\gamma_0$, this extends to an automorphism of $\tD_\eta$, which we will also denote by $\tauf$.
So} on the right hand side of \eqref{eq:Wedderburn-eta}, we have an entry-wise action of $\tauf$. We shall now relate these two actions.
Let $x$ resp. $X$ be elements corresponding to each other under the Wedderburn isomorphism \eqref{eq:Wedderburn-eta}:
\begin{align*}
	\Q^ F(\Gamma_0)[H]\epsilon(\eta) &\simeq M_{\nFeta}(\tD_\eta) \\
	x &\leftrightarrow X
\end{align*}
We adopt the convention of denoting elements corresponding to each other by the same letter, lowercase on the left and uppercase on the right.

To emphasise where each action is coming from, we introduce the following notation.
We write $\gamma^{\vf} x\gamma^{-\vf}={}^{\gamma^{\vf}} x$ for the $\gamma^{\vf}$-conjugate on the left, and $\pc(X)$ for the corresponding element on the right. Similarly, on the right hand side we write $\tauf(X)$ for the matrix obtained from $X$ by applying $\tauf$ entry-wise, and let $\pt(x)$ be the corresponding element on the left hand side. This defines automorphisms on both sides of \eqref{eq:Wedderburn-eta}:
\[\pc\in\Aut(M_{\nFeta}(\tD_\eta)), \quad \pt\in\Aut(\Q^ F(\Gamma_0)[H]\epsilon(\eta))  .\]

\begin{remark} \label{rem:things-happen-over-QpH}
	The results of this section are in fact valid over $F[H]$, that is, before tensoring with $\Q^ F(\Gamma_0)$: indeed, since $\Gamma_0$ is central, conjugation acts trivially, and $\tauf$ was extended so that $\tauf(\gamma_0)=\gamma_0$. We nevertheless state everything in terms of $\Q^ F(\Gamma_0)$-tensored algebras, as this is the form in which these results will be used in the sequel.
\end{remark}

\begin{proposition} \label{prop:trivial-on-centre}
	The two actions just defined agree on the respective centres. In formul\ae: on the group ring side,
	\begin{align*}
		{}^{\gamma^{\vf}}(-)\big|_{\cent(\Q^ F(\Gamma_0)[H]\epsilon(\eta))} &= \pt\big|_{\cent(\Q^ F(\Gamma_0)[H]\epsilon(\eta))}  ,
		\intertext{and equivalently on the matrix ring side,}
		\pc\big|_{\cent\left(M_{\nFeta}(\tD_\eta)\right)} &= \tauf\big|_{\cent\left(M_{\nFeta}(\tD_\eta)\right)}  .
	\end{align*}
\end{proposition}
\begin{proof}
	By definition of $\pt$ and $\pc$, the two statements are equivalent, so it is enough to prove one of them: we shall prove the former. In the proof, we will work over the group ring $ F[H]$ instead of $\Q^ F(\Gamma_0)[H]$, which is permissible by \cref{rem:things-happen-over-QpH}.
	
	Consider the following commutative diagram, which we explain below.
	\[
	\begin{tikzcd}[column sep=5em]
		 F[H]\epsilon(\eta) \ar[r, "\sim"] \ar[d, hook] & M_{\nFeta}(D_\eta) \ar[d, hook] \\
		 F^\al[H]\epsilon(\eta) \ar[d, Rightarrow, no head] \ar[r, "\sim"] & \displaystyle\bigoplus_{\sigma \in \Gal( F(\eta)/ F)} M_{\eta(1)}( F^\al) \ar[d, Rightarrow, no head] \\
		\displaystyle\bigoplus_{\sigma\in\Gal( F(\eta)/ F)}  F^\al[H] e({}^\sigma\eta) \ar[r,"\sim", "\bigoplus_\sigma \rho_{{}^\sigma\eta}"'] & \displaystyle\bigoplus_{\sigma \in \Gal( F(\eta)/ F)} M_{\eta(1)}( F^\al) E({}^\sigma\eta)
	\end{tikzcd}
	\]
	The first row is the Wedderburn isomorphism \eqref{eq:Wedderburn-eta-QQ}. The second row is obtained by tensoring with $ F^\al$ over $F$; recall that $ F^\al$ is a splitting field for $D_\eta$, and that $\eta(1)=\nFeta \sFeta$. Since the tensor product is taken over $ F$ and not over $ F(\eta)$, we get a component for each embedding $ F(\eta)\hookrightarrow F^\al$.
	The vertical map on the left is inclusion. On the right, the map is induced by entry-wise application of the embedding 
	\begin{align*}
		D_\eta&\hookrightarrow D_\eta\otimes_{ F(\eta),\sigma} F^\al\simeq M_{\eta(1)}( F^\al)\\
		x&\mapsto x\otimes 1
	\end{align*}
	in each component.
	The third row of the diagram is induced from the second row by the decomposition of $\epsilon(\eta)$ into components $e({}^\sigma\eta)$. 
	Let $E({}^\sigma\eta)\in M_{\eta(1)}(F^\al)$ {be the matrix in the $\sigma$-component of} the image of $e({}^\sigma\eta)$ under the isomorphism in the second row; {in other words,} $E({}^\sigma\eta)$ is the identity matrix in the $\sigma$-component.
	We write 
	\[\rho_{{}^\sigma\eta}:  F^\al[H] e({}^\sigma\eta)\xrightarrow{\sim} M_{\eta(1)}( F^\al) E({}^\sigma\eta) {=} M_{\eta(1)}( F^\al)\]
	for the $\sigma$-part of the map in the third row. This is a representation of $H$ with character ${}^\sigma\eta$.
	
	Let $z\in \cent( F[H]\epsilon(\eta))$ be a central element. Then for all $\sigma\in\Gal( F(\eta)/ F)$,
	\[\rho_{{}^\sigma\eta}(z)=\frac{1}{\eta(1)} \Tr\left(\rho_{{}^\sigma\eta}(z)\right)\cdot E\left({}^\sigma\eta\right)=\frac{1}{\eta(1)} {}^\sigma\eta(z)\cdot E\left({}^\sigma\eta\right)  ,\]
	where by abuse of notation, ${}^\sigma\eta$ denotes the $ F^\al$-linear extension of the character ${}^\sigma\eta:H\to F^\al$ to the group ring $ F^\al[H]$.
	
	A central element of $M_{\nFeta}(D_\eta)$ is of the form $Z_\alpha=\alpha \mathbf 1_{\nFeta}$ where $\alpha\in \cent(D_\eta)= F(\eta)$. Let $z_\alpha\in \cent( F[H]\epsilon(\eta))$ be the corresponding central element in the group ring under the top horizontal map. By commutativity of the diagram, the images of $Z_\alpha$ and $z_\alpha$ in the bottom right corner coincide:
	\begin{equation*}
		\sum_{\sigma\in\Gal( F(\eta)/ F)} \sigma\left(Z_\alpha\right) E\left({}^\sigma\eta\right) = \sum_{\sigma\in\Gal( F(\eta)/ F)} \frac{1}{\eta(1)}{}^\sigma\eta(z_\alpha)E\left({}^\sigma\eta\right)  .
	\end{equation*}
	Using this together with the fact that $F(\eta)/F$ is abelian, we compute the image of ${}^{\gamma^{\vf}}z_\alpha$:
	\begin{align*}
		\sum_{\sigma\in\Gal( F(\eta)/ F)} \frac{1}{\eta(1)}{}^\sigma\eta\left({}^{\gamma^{\vf}}z_\alpha\right)E\left({}^\sigma\eta\right) = \sum_{\sigma\in\Gal( F(\eta)/ F)} \sigma\left(Z_{\tauf(\alpha)}\right) E\left({}^\sigma\eta\right){.}
	\end{align*}
	The right hand side is the image of $Z_{\tauf(\alpha)}$. The corresponding group ring element is $z_{\tauf(\alpha)}$, which, by definition of $\pt$, is the same as $\pt(z_\alpha)$. This concludes the proof.
\end{proof}

\Cref{prop:trivial-on-centre} shows that the automorphism $\pt^{-1}\circ{}^{\gamma^{\vf}}(-)$ of the central simple $\Q^{ F(\eta)}(\Gamma_0)$-algebra $\Q^ F(\Gamma_0)[H]\epsilon(\eta)$ is trivial on the centre. Pursuant to the Skolem--Noether theorem, see \cite[Theorem~3.62]{CR}, there is a unit $\yFeta\in \Q^ F(\Gamma_0)[H]\epsilon(\eta)^\times$ such that for all $x\in\Q^ F(\Gamma_0)[H]\epsilon(\eta)$,
\begin{align}
	\pt^{-1} \left({}^{\gamma^{\vf}}x\right) &= \yFeta \cdot x \cdot \yFeta^{-1}  . \notag
	\intertext{Equivalently we may write}
	{}^{\gamma^{\vf}}x &= \pt(\yFeta) \cdot \pt (x) \cdot \pt(\yFeta^{-1})  . \label{eq:yeta}
	\intertext{The corresponding equations for matrices also hold: 
		if $\YFeta\in\GL_{\nFeta}(\tD_\eta)$ denotes the element corresponding to $\yFeta$ under \eqref{eq:Wedderburn-eta}, then for all $X\in M_{\nFeta}(\tD_\eta)$,}
	\pc(X)&=\tauf(\YFeta) \cdot \tauf(X) \cdot \tauf(\YFeta^{-1})  . \label{eq:Yeta}
	\intertext{By \cref{rem:things-happen-over-QpH}, we may in fact assume $\yFeta\in  F[H]\epsilon(\eta)^\times$ and $\YFeta\in \GL_{\nFeta}(D_\eta)$. It should be noted that the Skolem--Noether theorem only determines the units $\yFeta$ and $\YFeta$ up to central units in their respective ambient rings.
		By induction, \eqref{eq:yeta} and \eqref{eq:Yeta} admit the following generalisations: for all $\vf \mid i$,}
	{}^{\gamma^i}x &= \pt(\yFeta) \cdots \pt^{i/\vf}(\yFeta) \cdot \pt^{i/\vf} (x) \cdot \pt^{i/\vf}\left(\yFeta^{-1}\right)\cdots \pt\left(\yFeta^{-1}\right)  , \label{eq:yeta-higher} \\
	\pc^{i/\vf}(X)&=\tauf(\YFeta)\cdots \tauf^{i/\vf}(\YFeta) \cdot \tauf^{i/\vf}(X) \cdot \tauf^{i/\vf}\left(\YFeta^{-1}\right)\cdots \tauf\left(\YFeta^{-1}\right)  . \label{eq:Yeta-higher}
\end{align}
We shall write 
\begin{equation}
    A_{i/\vf}\colonequals \tauf(\YFeta)\cdots \tau^{i/\vf}(\YFeta)\in\GL_{\nFeta}(D_\eta)
\end{equation}
for the conjugating element in \eqref{eq:Yeta-higher}, and 
\begin{equation}
    \af_{i/\vf}\colonequals \pt(\yFeta) \cdots \pt^{i/\vf}(\yFeta)\in  F[H]\epsilon(\eta)^\times
\end{equation}
for the corresponding element in \eqref{eq:yeta-higher}. 
Here we work with a fixed character $\eta$, which is why it's convenient to suppress it from the notation $a_{i/v_\chi}$.
It follows from the definitions that if $\vf\mid i,j$, then 
\begin{equation} \label{eq:AiAj}
	A_{i/\vf}\cdot\tau^{i/\vf} (A_{j/\vf})=A_{(i+j)/\vf}  ,
\end{equation} 
and similarly 
\begin{equation} \label{eq:aiaj}
	\af_{i/\vf}\cdot\pt^{i/\vf} (\af_{j/\vf})=\af_{(i+j)/\vf}  .
\end{equation}
The element $\af_{p^{n_0}/\vf}$ is central in $\Q^ F(\Gamma_0)[H]\epsilon(\eta)$. Indeed, for all $x\in\Q^ F(\Gamma_0)[H]\epsilon(\eta)$:
\begin{equation} \label{eq:a-pn0-central}
	x={}^{\gamma_0} x = {}^{\gamma^{p^{n_0}}} x = \af_{p^{n_0}/\vf} \cdot \pt^{p^{n_0}/\vf} (x) \cdot \left(\af_{p^{n_0}/\vf}\right)^{-1}=\af_{p^{n_0}/\vf} \cdot x \cdot \left(\af_{p^{n_0}/\vf}\right)^{-1}  .
\end{equation}

\begin{definition}\label{def:gpc}
	Let $\gcp\colonequals (\af_1)^{-1}\gamma^{\vf}=\pt\left(\yFeta^{-1}\right)\gamma^{\vf}$, and let $\Gcp$ be the procyclic group generated by $\gcp$.
\end{definition}

\begin{lemma} \label{power-of-gcp}
	For $j\ge 1$, we have $(\gcp)^j=a_j^{-1} \gamma^{j\vf}$.
\end{lemma}
\begin{proof}
	Follows from \eqref{eq:aiaj} by induction.
\end{proof}

{
Similarly to \cref{changing-v}, we record how these elements vary under replacing $F$ by some intermediate field of $F(\eta)/F_\chi$. Recall that $\Gal(F(\eta)/E)=\langle\tau^f\rangle$ if $f=(E:F_\chi)$. To emphasise dependence on the base field, we write $\delta^F_\tau$ resp. $\gcpf$ for $\delta_\tau$ resp. $\gcp$.
\begin{lemma} \label{lem:gamma-chi-W}
    Let $F_\chi \subseteq E\subseteq F(\eta)$ be an intermediate field, and let $f\colonequals(E:F_\chi)$. Then $\delta^E_{\tau^f}=(\delta^F_\tau)^f$, and one can choose $y_\eta^E$ such that $\gamma^{E,\prime\prime}_\eta=(\gcpf)^f$.
\end{lemma}
\begin{proof}
    The assertion on $\delta^E_{\tau^f}$ is a direct consequence of the definitions and \cref{changing-v}. For the second assertion, recall the defining properties of $\delta^F_\tau$ resp. $\delta^E_{\tau^f}$:
    \begin{align}
        {}^{\gamma^{v_\chi^F}} x &= \delta_\tau^F\left( y_\eta^F \cdot x\cdot (y_\eta^F)^{-1}\right) & \forall x\in \Q^F(\Gamma_0)[H]\epsilon^F(\eta), \label{eq:v-chi-F-defining-property} \\
        {}^{\gamma^{v_\chi^E}} x &= \delta_{\tau^f}^E\left( y_\eta^E \cdot x\cdot (y_\eta^E)^{-1}\right) & \forall x\in \Q^E(\Gamma_0)[H]\epsilon^E(\eta) . \notag
    \end{align}
    Applying \eqref{eq:v-chi-F-defining-property} $f$ times, and using $v_\chi^E=f v_\chi^F$ from \cref{changing-v}, we obtain
    \[{}^{\gamma^{v_\chi^E}} x = \delta_\tau^F\!\left( y_\eta^F \right) \cdot \left(\delta_\tau^F\right)^2\!\left( y_\eta^F \right)\cdot \ldots \cdot \left(\delta_\tau^F\right)^f\!\left( y_\eta^F \right) \cdot \left(\delta_\tau^F\right)^f\!( x) \cdot \left(\delta_\tau^F\right)^f\!\left( y_\eta^F \right)^{-1}\cdot\ldots\cdot \delta_\tau^F\!\left( y_\eta^F \right)^{  {-1}}.\]
    Therefore if we first choose $y_\eta^F$, then
    \[y_\eta^E\colonequals \left(\delta_{\tau^f}^E\right)^{-1}\!\left( \delta_\tau^F\!\left( y_\eta^F \right) \cdot \ldots \cdot \left(\delta_\tau^F\right)^f\!\left( y_\eta^F \right) \right)\]
    is a suitable choice for $y_\eta^E$. With these choices, \cref{power-of-gcp} shows that $\gamma^{E,\prime\prime}_\eta=(\gcpf)^f$.
\end{proof}}

\subsection{Conjugating indecomposable idempotents} \label{sec:conjugating-f-idempotents}
{The primitive central idempotent $\epsilon(\eta)$ is the sum of indecomposable idempotents $f_\eta^{(j)}$ introduced in \cref{def:f-idempotent}.}
Later it will be necessary to keep track of what happens to the idempotents $f_\eta^{(j)}$ under conjugation. To ease notation, we restrict our attention to delineating the behaviour of $f_\eta^{(1)}$ with respect to conjugation: the general case of $f_\eta^{(j)}$ is completely analogous.

We shall need the following observation from linear algebra. Let $ D$ be a skew field, and let $n\ge1$. Let $\Pi: M_n( D)\to  D$ be the map sending an $n\times n$ matrix $(x_{i,j})$ to its $(1,1)$-entry $x_{1,1}$. The map $\Pi$ is additive and $ D$-linear: in other words, it is a $ D$-vector space homomorphism, where $ D$ acts on $M_n( D)$ by left multiplication. Note that $\Pi$ fails to be multiplicative in general.

{Let $E_1\colonequals \diag(1,0,\ldots,0)$.} Consider the identity
\[
{E_1}
\begin{psmallmatrix}
	x_{1,1} & \dots & x_{1,n} \\ \vdots &\ddots & \vdots\\ x_{n,1} & \dots & x_{n,n}
\end{psmallmatrix}
{E_1}
=
{x_{1,1} E_1}.
\]
On the subspace of $n\times n$ matrices of this form, $\Pi$ becomes a $ D$-vector space isomorphism with respect to left $ D$-multiplication:
\begin{equation} \label{eq:pi-vs-iso}
	\Pi: {E_1} M_n( D) {E_1} \xrightarrow{\sim}  D  .
\end{equation}
On this subspace, $\Pi$ is also multiplicative: indeed, the product of two such matrices is again a matrix of this shape with the $(1,1)$-entries multiplied. That is, $\Pi$ is a ring isomorphism on this subspace.

{For an invertible matrix $A\in\GL_n(D)$, we define an idempotent
$e_A\colonequals A {E_1} A^{-1}.$
The pieces of $M_n(D)$ cut out by left and right multiplication by such idempotents obey the following rules.}

\begin{lemma} \label{f-conjugate}
	Let $A, B\in \GL_n( D)$. Then there is a well-defined isomorphism of $\cent( D)$-vector spaces:
	{\begin{align*}
		\Pi_{A,B}:e_A M_n( D) e_B &\xrightarrow{\sim}  D \\
		e_A X e_B &\mapsto \Pi\left({E_1} A^{-1} X B {E_1} \right) \end{align*}}
\end{lemma}
\begin{proof}
	The assertion is clear from \eqref{eq:pi-vs-iso}: the map $\Pi_{A,B}$ is  $\Pi$ precomposed with {left multiplication by $A^{-1}$ and right multiplication by $B$.}
\end{proof}

\begin{lemma} \label{Pi-AB-multiplication-rule}
	Let $A,B,C\in \GL_n( D)$. Then there is the following multiplication rule: for all $X,Y\in M_n( D)$, one has {$\Pi_{A,B}\left(e_A X e_B\right) \cdot \Pi_{B,C}\left(e_B Y e_C\right) = \Pi_{A,C}\left(e_A X e_B \cdot e_B Y e_C\right)$}.
\end{lemma}
\begin{proof}
	The statement is a direct consequence of the definitions.
\end{proof}

Let us now specialise to the case $ D\colonequals\tD_\eta$. Passing through the Wedderburn isomorphism \eqref{eq:Wedderburn-eta}, the ring isomorphism $\Pi$ in \eqref{eq:pi-vs-iso} gives rise to a ring isomorphism
\begin{equation} \label{eq:pi-iso}
	\pi: f_\eta^{(1)} \Q^ F(\Gamma_0)[H]\epsilon(\eta)f_\eta^{(1)}=f_\eta^{(1)}\Q^ F(\Gamma_0)[H]f_\eta^{(1)}\xrightarrow{\sim} \tD_\eta  .
\end{equation}

Let $a,b\in \Q^ F(\Gamma_0)[H]\epsilon(\eta)^\times$ be two units corresponding to $A,B\in\GL_n(\tD_\eta)$. By pre- resp. postcomposing $\Pi_{A,B}$ with the Wedderburn isomorphism \eqref{eq:Wedderburn-eta} resp. its inverse, we define maps $\pi_{a,b}$. These are $\cent(\tD_\eta)$-vector space isomorphisms (\cref{f-conjugate}) satisfying a multiplication rule analogous to \cref{Pi-AB-multiplication-rule}.

For the conjugation action studied in \cref{sec:deltas}, we obtain the following:
\begin{lemma}\label{fQpHf}
	Let $\vf\mid i,j$. Then there are $\cent(\tD_\eta)$-vector space isomorphisms
	\begin{align*}
		\Pi_{i,j}: \pc^{i/\vf} \left({E_1}\right) \cdot M_{\nFeta}(\tD_\eta) \cdot \pc^{j/\vf}\left({E_1}\right) &\xrightarrow{\sim} \tD_\eta  , \\
		\pi_{i,j}: {}^{\gamma^{i}}f_\eta^{(1)}\cdot \Q^ F(\Gamma_0)[H]\epsilon(\eta) \cdot {}^{\gamma^j} f_\eta^{(1)} &\xrightarrow{\sim} \tD_\eta  .
	\end{align*}
	Moreover, $\Pi=\Pi_{0,0}$ and $\pi=\pi_{0,0}$. \qed
\end{lemma}

\begin{remark}
	The conjugate ${}^{\gamma^\ell} f_\eta^{(j)} = \gamma^\ell f_\eta^{(j)} \gamma^{-\ell}$ is also an indecomposable idempotent, but it need not be of the form $f_\eta^{(k)}$ for some $k$. As illustrated by the following example, this is not even the case for the corresponding matrices:
	\[
	\begin{psmallmatrix}
		2 & 3 \\ 1 & 2
	\end{psmallmatrix}
	\begin{psmallmatrix}
		1 & 0 \\ 0 & 0
	\end{psmallmatrix}
	\begin{psmallmatrix}
		2 & 3 \\ 1 & 2
	\end{psmallmatrix}^{-1} = 
	\begin{psmallmatrix}
		2 & 3 \\ 1 & 2
	\end{psmallmatrix}
	\begin{psmallmatrix}
		1 & 0 \\ 0 & 0
	\end{psmallmatrix}
	\begin{psmallmatrix}
		2 & -3 \\ -1 & 2
	\end{psmallmatrix} =
	\begin{psmallmatrix}
		4 & -6 \\ 2 & -3
	\end{psmallmatrix}  . \qedhere
	\]
\end{remark}

\subsection{Ring structure of \texorpdfstring{$f_\eta^{(1)} \Q^ F(\G)f_\eta^{(1)}$}{f1η QF(G) f1η}} \label{sec:ring-structure}
The module $f_\eta^{(1)} \Q^ F(\G)f_\eta^{(1)}$ is a ring: indeed, if $x,y\in\Q^ F(\G)$, then addition and multiplication rules in $f_\eta^{(1)} \Q^ F(\G)f_\eta^{(1)}$ are as follows:
\begin{align*}
	f_\eta^{(1)} x f_\eta^{(1)}+f_\eta^{(1)}yf_\eta^{(1)}&=f_\eta^{(1)}(x+y)f_\eta^{(1)}  , \\
	f_\eta^{(1)} xf_\eta^{(1)}\cdot f_\eta^{(1)}yf_\eta^{(1)}&=f_\eta^{(1)}\big(xf_\eta^{(1)}y\big)f_\eta^{(1)}  .
\end{align*}
The unity element is $f_\eta^{(1)}$. We will now describe the multiplication rule in more detail.

Under the decomposition \eqref{eq:Q-decomposition}, an element $x\in\Q^ F(\G)$ can be written as
\begin{equation} \label{eq:x-Q-decomp}
	x=\sum_{\ell=0}^{p^{n_0}-1} x_\ell \gamma^\ell, \quad x_\ell\in\Q^ F(\Gamma_0)[H]  ,
\end{equation}
and similarly for $y$.
When considering $f_{\eta}^{(1)} x f_{\eta}^{(1)}$, we may, without loss of generality, restrict the summation to indices $\ell$ divisible by $\vf$: this is a consequence of orthogonality of the $\epsilon$-idempotents of $F[H]$, similarly to \cite[1224]{Lau} and \cite[610]{NickelConductor}:
\begin{noproof}{lemma} \label{divisibility-by-v_chi}
	We have $f_{\eta}^{(1)} x_\ell \gamma^\ell f_{\eta}^{(1)} = 0$ whenever $\vf\nmid \ell$. Therefore
	\[f_{\eta}^{(1)} x f_{\eta}^{(1)} = \sum_{\ell=0}^{p^{n_0}-1} f_{\eta}^{(1)} x_\ell \gamma^\ell f_{\eta}^{(1)} = \sum_{\substack{\ell=0 \\ \vf\mid \ell}}^{p^{n_0}-1} f_{\eta}^{(1)} x_\ell \gamma^\ell f_{\eta}^{(1)}  . \qedhere\]
\end{noproof}

\Cref{divisibility-by-v_chi} allows us {to} compute products in $f_\eta^{(1)}\Q^ F(\G)f_\eta^{(1)}$: since all powers of $\gamma$ occurring in the sum are divisible by $\vf$, conjugation by them acts as some power of $\pt$ as in \eqref{eq:yeta-higher}.
\begingroup\allowdisplaybreaks
\begin{align}
	&\phantom{==}f_\eta^{(1)} xf_\eta^{(1)}\cdot f_\eta^{(1)}yf_\eta^{(1)} \label{eq:mult-rule-0} \\
	&= \sum_{\substack{\ell,\ell'=0 \\ \vf\mid \ell,\ell'}}^{p^{n_0}-1} f_{\eta}^{(1)} \cdot x_\ell \cdot {}^{\gamma^\ell}f_{\eta}^{(1)} \cdot {}^{\gamma^\ell}f_{\eta}^{(1)}\cdot {}^{\gamma^\ell}y_{\ell'}\cdot {}^{\gamma^{\ell+\ell'}}f_{\eta}^{(1)} \cdot \gamma^{\ell+\ell'} \notag
	\intertext{Expand the conjugation action by using \eqref{eq:yeta-higher} {and \eqref{eq:aiaj}}. Since $\tauf$ acts trivially on $\diag(1,0\ldots,0)$, the corresponding automorphism $\pt$ also acts trivially on $f_\eta^{(1)}$ by definition.}
	&= \sum_{\substack{\ell,\ell'=0 \\ \vf\mid \ell,\ell'}}^{p^{n_0}-1} 
	f_{\eta}^{(1)} x_\ell  \af_{\ell/\vf} f_{\eta}^{(1)} \cdot 
	f_{\eta}^{(1)} \pt^{\ell/\vf}(y_{\ell'}) \pt^{\ell/\vf}(a_{\ell'/\vf}) f_{\eta}^{(1)} 
	\cdot \left(\af_{(\ell+\ell')/\vf}\right)^{-1} \cdot \gamma^{\ell+\ell'} \label{before-fQGf-multiplication-rule}
\end{align}\endgroup

In \eqref{before-fQGf-multiplication-rule}, we have a description of the general multiplication rule of the ring $f_\eta^{(1)} \Q^ F(\G)f_\eta^{(1)}$. Notice that this is controlled by the automorphism $\pt$. It follows from \eqref{eq:yeta} that the element $\gcp$, which we have already singled out in \cref{def:gpc}, plays a special r\^ole:

\begin{noproof}{lemma} \label{gcp-conjugation}
	Conjugation by $\gcp$ acts as $\pt$ on $f_{\eta}^{(1)} \Q^ F(\Gamma_0)[H] f_\eta^{(1)}$, that is, for all $y_0\in \Q^ F(\Gamma_0)[H]$,
    \begin{align*}
		\gcp\cdot f_\eta^{(1)} y_0 f_\eta^{(1)}\cdot (\gcp)^{-1} 
		&= f_\eta^{(1)} \pt(y_0) f_\eta^{(1)}  . \qedhere
	\end{align*}
\end{noproof}

\section{The totally ramified case} \label{sec:TORC}

\subsection{Indecomposability of group ring idempotents}\label{sec:fQGf-skew-field}
Let $\eta$ be an irreducible character of $H$.
The element $f_{\eta}^{(j)}$ is then an indecomposable idempotent in the group ring $ F[H]$, and thus 
$f_{\eta}^{(j)}  F[H] f_{\eta}^{(j)}\simeq D_{\eta}$, as witnessed by \eqref{Deta-feta}. 
This section is devoted to the proof of the following:
\begin{theorem} \label{fQGf-skew-field} \label{thm:skew-field-totally-ramified}
	Let $\eta\in\Irr(H)$ be an irreducible constituent of $\res^\G_H\chi$, and let $1\le j\le \nFeta$. Assume that the extension $ F(\eta)/ F_\chi$ is totally ramified. Then the algebra $f_{\eta}^{(j)} \Q^ F(\G) f_{\eta}^{(j)}$ is a skew field.
\end{theorem}
In other words, \cref{fQGf-skew-field} states that the idempotent $f_{\eta}^{(j)}$ remains indecomposable in $\Q^ F(\G)$ {\cite[Proposition~3.18]{CR}}.
For the rest of this section, we fix $\chi$ and $\eta$ as in the statement, and assume that $ F(\eta)/ F_\chi$ is totally ramified. Without loss of generality, we may restrict our attention to the case $j=1$.

\Cref{fQGf-skew-field} states that every nonzero element of $f_\eta^{(1)}\Q^ F(\G)f_\eta^{(1)}$ admits a left and right inverse. 
Since the two are equivalent, we shall work with left inverses, and move powers of $\gamma$ to the right: this will make the formul\ae{} arising from the multiplication rule slightly more palatable.
Let $x\in\Q^ F(\G)$ be such that $f_\eta^{(1)}xf_\eta^{(1)}\ne0$. We seek a $y\in\Q^ F(\G)$ such that 
\begin{equation} \label{eq:big-equation}
	f_\eta^{(1)} y f_\eta^{(1)} \cdot f_\eta^{(1)} x f_\eta^{(1)}= f_\eta^{(1)} \Big( 1\cdot\gamma^0+0\cdot\gamma^1+\ldots+0\cdot\gamma^{p^{n_0}-1} \Big) f_\eta^{(1)} = f_\eta^{(1)}  .
\end{equation}
Under the decomposition \eqref{eq:Q-decomposition}, the elements $x$ resp. $y$ can be written as
\begin{equation*}
	x=\sum_{\ell'=0}^{p^{n_0}-1} x_{\ell'} \gamma^{\ell'} \quad\text{resp.}\quad y=\sum_{\ell=0}^{p^{n_0}-1} y_{\ell} \gamma^{\ell}, \quad x_{\ell'},y_\ell\in\Q^ F(\Gamma_0)[H]  .
\end{equation*}
Using \cref{divisibility-by-v_chi} and the multiplication rule \eqref{before-fQGf-multiplication-rule}, the equation \eqref{eq:big-equation} can be rewritten as
\begin{equation*}
	\sum_{\substack{\ell,\ell'=0 \\ \vf\mid \ell,\ell'}}^{p^{n_0}-1}  
	\left(f_\eta^{(1)} y_\ell  \af_{\ell/\vf} f_\eta^{(1)}\right)\cdot 
	\left(f_\eta^{(1)} \pt^{\ell/\vf}(x_{\ell'} \af_{\ell'/\vf}) f_\eta^{(1)}\right) \cdot (\af_{(\ell+\ell')/\vf})^{-1} \cdot \gamma^{\ell+\ell'} = f_\eta^{(1)} \gamma^0  f_\eta^{(1)}  .
\end{equation*}
This is equivalent to the collection of the following equations for $0\le k< p^{n_0}$ and $\vf\mid k$:
\[\sum_{\substack{\ell,\ell'=0 \\ \ell+\ell'\equiv k \pod{p^{n_0}} \\ \vf\mid \ell,\ell'}}^{p^{n_0}-1}  
\left(f_\eta^{(1)} y_\ell  \af_{\ell/\vf} f_\eta^{(1)}\right)\cdot 
\left(f_\eta^{(1)} \pt^{\ell/\vf}(x_{\ell'} \af_{\ell'/\vf}) f_\eta^{(1)}\right) \cdot (\af_{(\ell+\ell')/\vf})^{-1} \cdot \gamma_0^t \gamma^k = f_\eta^{(1)}\delta_{0,k} \gamma^k f_\eta^{(1)}  .\]
Here $\delta_{0,k}$ is the Kronecker delta, and $t=t(k,\ell)$ is defined by the equation $\ell+\ell'=k+p^{n_0}t$, that is, $t=1$ if $k<\ell$ (or equivalently $k<\ell'$) and zero otherwise; {we will often omit the arguments $k$ and $\ell$}. On the level of coefficients, we have an equation for each $k$:
\[\sum_{\substack{\ell,\ell'=0 \\ \ell+\ell'\equiv k \pod{p^{n_0}} \\ \vf\mid \ell,\ell'}}^{p^{n_0}-1}  
\left(f_\eta^{(1)} y_\ell  \af_{\ell/\vf} f_\eta^{(1)}\right)\cdot 
\left(f_\eta^{(1)} \pt^{\ell/\vf}(x_{\ell'} \af_{\ell'/\vf}) f_\eta^{(1)}\right) \cdot \left(\af_{(k+p^{n_0}t)/\vf}\right)^{-1} \cdot \gamma_0^t = \delta_{0,k} f_\eta^{(1)}  . \]
Consider the factor $a_{(k+p^{n_0}t)/\vf}^{-1}$: using \eqref{eq:aiaj}, it can be rewritten as {$\pt^{k/\vf}(a_{p^{n_0}t/\vf})^{-1}a_{k/\vf}^{-1}$}. The factor $a_{k/\vf}^{-1}$ depends only on $k$, and thus can be removed by multiplying the $k$th equation by $a_{k/\vf}$. For $k=0$, this is just $1$, and for $k\ne0$, the right hand side is zero: in either case, the right hand side does not change.
The factor {$\pt^{k/\vf}(a_{p^{n_0}t/\vf})^{-1}$} is central in $ F(\eta)[H]\epsilon(\eta)$ by \eqref{eq:a-pn0-central}, so it may be moved inside the second factor in brackets.

In conclusion, applying the isomorphism $\pi$ to the multiplication rule \eqref{before-fQGf-multiplication-rule}, we find that \eqref{eq:big-equation} is equivalent to the following system of linear equations over $\tD_\eta$, with indeterminates $\pi(f_\eta^{(1)} y_\ell  \af_{\ell/\vf} f_\eta^{(1)})$:
\begin{equation} \label{eq:sys-lin}
	\sum_{\substack{\ell,\ell'=0 \\ \ell+\ell'\equiv k \pod{p^{n_0}} \\ \vf\mid \ell,\ell'}}^{p^{n_0}-1}  
	\underbrace{\pi\left(f_\eta^{(1)} y_\ell  \af_{\ell/\vf} f_\eta^{(1)}\right)}_{\in \tD_\eta}
	\cdot \underbrace{\pi\left(f_\eta^{(1)} \pt^{\ell/\vf}(x_{\ell'}a_{\ell'/\vf}) \pt^{k/\vf}(a_{p^{n_0}t/\vf})^{-1} f_\eta^{(1)}\right) }_{d_{k,\ell} \in \tD_\eta} \gamma_0^t
	= \delta_{0,k}  .
\end{equation}
Here $k$ runs over the numbers $0\le k<p^{n_0}$ such that $\vf\mid k$.
The two factors in the summation are both in $\tD_\eta$ due to \cref{fQpHf}, with the product operation being the one of $\tD_\eta$ (see \cref{Pi-AB-multiplication-rule}). Note that the second factor inside the summation, denoted by $d_{k,\ell}$, depends only on $k$ and $\ell$, since $\ell'$ and $t$ are determined by these two.

{We conclude that} \eqref{eq:big-equation} admits a solution if and only if the system of linear equations \eqref{eq:sys-lin} over $\tD_\eta$, with $0\le k<p^{n_0}$ and $\vf\mid k$, has a solution.
{Indeed, this is true since $\pi:f_\eta^{(1)}\Q^F(\Gamma_0)[H]f_\eta^{(1)}\xrightarrow{\sim}\tD_\eta$ is an isomorphism of rings, and each $a_{\ell/v_\chi}$ is a unit in $F[H]\epsilon(\eta)\subseteq \Q^F(\Gamma_0)[H]\epsilon(\eta)$.}

Let $n\colonequals p^{n_0}/\vf$, and let $M$ be the $n\times n$ matrix whose $(i,j)$th entry is
\[
d_{(i-1)\vf,(j-1)\vf}\gamma_0^{ t} =
\begin{cases}
	d_{(i-1)\vf,(j-1)\vf}\gamma_0 & \text{if }i<j  , \\
	d_{(i-1)\vf,(j-1)\vf} & \text{if }i\ge j  . \\
\end{cases}
\]
{Here $t=t\left((i-1)\vf,(j-1)\vf\right)$.}
The existence of a {unique} solution to the system of linear equations \eqref{eq:sys-lin} is thus equivalent to the matrix $M$ being nonsingular.

Given a pair $(\ell,k)$ in the summation in \eqref{eq:sys-lin}, the number $\ell'$ is uniquely determined. Moreover, the pairs $(\ell,k)$ and $(\ell+\vf,k+\vf)$ yield the same $\ell'$, since $\ell+\ell'\equiv k \pmod{p^{n_0}}$.
Therefore
\begin{align}
	d_{k,\ell} &= \pi\left(f_\eta^{(1)} \pt^{\ell/\vf}(x_{\ell'}a_{\ell'/\vf}) \pt^{k/\vf}(a_{p^{n_0}t/\vf})^{-1} f_\eta^{(1)}\right)  , \label{eq:dkl} \\
	d_{k+\vf,\ell+\vf} &= \pi\left(f_\eta^{(1)} \pt^{(\ell+\vf)/\vf}(x_{\ell'}a_{\ell'/\vf}) \pt^{(k+\vf)/\vf}(a_{p^{n_0}t/\vf})^{-1} f_\eta^{(1)}\right)  . \label{eq:dkl+}
\end{align}
Recall that $\pt$ comes from the entry-wise $\tauf$-action via the Wedderburn isomorphism, and that $\pi$ comes from the map $\Pi$ via the Wedderburn isomorphism. Thus \eqref{eq:dkl} and \eqref{eq:dkl+} allow us to deduce that 
\begin{equation} \label{eq:taudkl}
	\tauf(d_{k,\ell})=d_{k+\vf,\ell+\vf}  ,
\end{equation}
where the indices are understood modulo $p^{n_0}$.

Let us write $m_i\colonequals d_{0,i\vf}$ for $0\le i<n$. The condition $f_\eta^{(1)}xf_\eta^{(1)}\ne0$ implies that at least one of these $m_i$s is not zero.
Then using \eqref{eq:taudkl}, we find that $M\in M_n(\tD_\eta)$ is of the following shape; recall that $\tauf$ acts trivially on $\gamma_0$ by definition (see \cref{sec:deltas}):
\begin{equation} \label{M-shape}
	M=\begin{pmatrix}
		m_0 & m_1\gamma_0 & m_2\gamma_0 & \cdots & m_{n-2}\gamma_0 & m_{n-1}\gamma_0 \\
		\tauf(m_{n-1}) & \tauf(m_0) & \tauf(m_1)\gamma_0 & \ddots & \tauf(m_{n-3})\gamma_0 & \tauf(m_{n-2})\gamma_0 \\
		\tauf^2(m_{n-2}) & \tauf^2(m_{n-1}) & \tauf^2(m_0) & \ddots & \tauf^2(m_{n-4})\gamma_0 & \tauf^2(m_{n-3})\gamma_0 \\
		\vdots & \vdots & \vdots & \ddots & \ddots & \vdots \\
		\tauf^{n-2}(m_2) & \tauf^{n-2}(m_3) & \tauf^{n-2}(m_4) & \cdots & \tauf^{n-2}(m_0) & \tauf^{n-2}(m_1)\gamma_0 \\
		\tauf^{n-1}(m_1) & \tauf^{n-1}(m_2) & \tauf^{n-1}(m_3) & \cdots & \tauf^{n-1}(m_{n-1}) & \tauf^{n-1}(m_0) \\
	\end{pmatrix}  .
\end{equation}
We conclude that the statement of \cref{fQGf-skew-field} is the special case of the following assertion:
let $M\in M_n(\tD_\eta)$ be as in \eqref{M-shape}, with at least one of $m_0,\ldots,m_{n-1}$ not zero; then $M$ is nonsingular.
In fact, since the $a$s are units and $\pi$ from \eqref{eq:pi-iso} is an isomorphism, any such matrix $M$ arises from some $x$, so this assertion on the nonsingularity of $M$ is equivalent to \cref{fQGf-skew-field}.

Let us write $m_i'\colonequals \tauf^i(m_{n-i})$ where $0\le i<n$ and $n-i$ is understood modulo $n$ (that is, to be $0$ when $i=0$). Then the transpose of $M$ is
\begin{equation} \label{eq:M'-shape}
	M^\top=\begin{pmatrix}
		m'_0 & m'_1 & m'_2 & \cdots & m'_{n-2} & m'_{n-1} \\
		\tauf(m'_{n-1})\gamma_0 & \tauf(m'_0) & \tauf(m'_1) & \ddots & \tauf(m'_{n-3}) & \tauf(m'_{n-2}) \\
		\tauf^2(m'_{n-2})\gamma_0 & \tauf^2(m'_{n-1})\gamma_0 & \tauf^2(m'_0) & \ddots & \tauf^2(m'_{n-4}) & \tauf^2(m'_{n-3}) \\
		\vdots & \vdots & \vdots & \ddots & \ddots & \vdots \\
		\tauf^{n-2}(m'_2)\gamma_0 & \tauf^{n-2}(m'_3)\gamma_0 & \tauf^{n-2}(m'_4)\gamma_0 & \cdots & \tauf^{n-2}(m'_0) & \tauf^{n-2}(m'_1) \\
		\tauf^{n-1}(m'_1)\gamma_0 & \tauf^{n-1}(m'_2)\gamma_0 & \tauf^{n-1}(m'_3)\gamma_0 & \cdots & \tauf^{n-1}(m'_{n-1})\gamma_0 & \tauf^{n-1}(m'_0) \\
	\end{pmatrix}  .
\end{equation}
Of course, $M$ is nonsingular if and only if $M^\top$ is, and the condition in the assertion above translates to at least one $m_i'$ not being zero.

Consider the following algebra: let
\[A\colonequals \bigoplus_{i=0}^{n-1} \tD_\eta \left(\gamma_0^{1/n}\right)^i\]
as a left $\tD_\eta$-module. Make $A$ an algebra by mandating the multiplication rule $\gamma_0^{1/n} d=\tauf(d)\gamma_0^{1/n}$ for all $d\in\tD_\eta$. {This makes sense since the order of $\tau$ divides $n$, or equivalently, the index $w_\chi=(\G:\G_\eta)$ of the stabiliser $\G_\eta$ of $\eta$ divides $p^{n_0}$, because $\gamma_0$ stabilises $\eta$.} Now the algebra $A$ has a basis over $\tD_\eta$ given by powers of $\gamma_0^{1/n}$. In this basis, right multiplication by the nonzero element $\sum_{i=0}^{n-1} m_i' (\gamma_0^{1/n})^i$ is given by the matrix $M^\top$ in \eqref{eq:M'-shape}. It follows that the assertion on the nonsingularity of $M$ is equivalent to $A$ being a skew field.

\begin{remark}
	In the definition of $A$, we consider $\gamma_0^{1/n}$ as a formal $n$th root of $\gamma_0$. By fiat, this has the same conjugation action as $\gamma^{\vf}$ does on $\eta$. However, since $n=p^{n_0}/\vf$, the identity $\gamma_0^{1/n}=\gamma^{\vf}$ holds in $\Gamma$. Therefore it behooves us to identify $\gamma_0^{1/n}$ with $\gamma^{\vf}$ for the rest of this section, with the tacit understanding that this is purely formal, and that things do not take place in $\Gamma$ or $\G$.
\end{remark}

We will express $A$ as a cyclic algebra. To this end, let 
\[\hD_\eta\colonequals D_\eta\otimes_{ F(\eta)} \Q^{ F(\eta)}\left(\Gamma_0^{\frac{w_\chi}{\vf n}}\right) = D_\eta\otimes_{ F(\eta)} \Q^{ F(\eta)}\left(\Gamma^{w_\chi}\right)  .\]
The cyclic algebra description \eqref{eq:D-cyclic-algebra} of $D_\eta$ provides a cyclic algebra description for $\hD_\eta$:
\[\hD_\eta=\bigoplus_{\ell=0}^{\sFeta-1} \Q^{ F(\eta)(\omega)}\left(\Gamma^{w_\chi}\right)\pi_{D_\eta}^\ell  .\]
{Here $\omega$ is a root of unity {of} order $q_\eta^{s_\eta}-1$, where $q$ is the order of the residue field of $F(\eta)$. We have $\pi_{D_\eta}^{s_\eta}=\pi_\eta$, and conjugation by $\pi_{D_\eta}$ acts via the automorphism $\sigma$ defined by $\sigma(\omega)=\omega^{q_\eta^{r_\eta}}$, where $r_\eta/s_\eta$ is the Hasse invariant of $D_\eta$.}
Therefore we may write $A$ as
\begin{equation}\label{eq:A-double}
	A= \bigoplus_{i=0}^{n-1} \tD_\eta (\gamma_0^{1/n})^i =  \bigoplus_{j=0}^{\frac{w_\chi}{\vf}-1} \hD_\eta (\gamma_0^{1/n})^j=
	\bigoplus_{j=0}^{\frac{w_\chi}{\vf}-1} \bigoplus_{\ell=0}^{\sFeta-1} \Q^{ F(\eta)(\omega)}\left(\Gamma^{w_\chi}\right)\pi_{D_\eta}^\ell \left(\gamma^{\vf}\right)^j  .
\end{equation}
Conjugation by $\pi_{D_\eta}$ acts as $\sigma$ whereas conjugation by $\gamma^{\vf}$ acts as $\tauf$. {We wish to combine the two sums into one: this is achieved by the following general lemma.}

\begin{lemma} \label{combining-double-sum}
    Let $\mathcal K/\mathcal F$ be a finite cyclic Galois extension, and let $\mathcal L_a$ and $\mathcal L_b$ be two subextensions such that $m_a\colonequals(\mathcal K:\mathcal L_a)$ and $m_b\colonequals(\mathcal K:\mathcal L_b)$ are coprime, $\mathcal L_a\mathcal L_b=\mathcal K$ {and $\mathcal L_a\cap \mathcal L_b=\mathcal F$}. In particular, $\Gal(\mathcal K/\mathcal F)=\Gal(\mathcal K/\mathcal L_a)\times\Gal(\mathcal K/\mathcal L_b)$. Let $\alpha$ resp. $\beta$ be generators of $\Gal(\mathcal K/\mathcal L_a)$ resp. $\Gal(\mathcal K/\mathcal L_b)$.
    
    Let $a$ be a nonzero algebraic element over {$\mathcal L_a$} such that $a^{m_a}\in \mathcal L_a^\times$, {and suppose that the cyclic algebra 
    \[\mathcal K(a)\colonequals(\mathcal K/\mathcal L_a,\alpha,a^{m_a})={\bigoplus_{\ell=0}^{m_a-1} \mathcal K a^\ell}\] has index $m_a$.}
    
    Suppose that $\beta$ extends to an automorphism {$\tilde \beta \in \Aut_{\mathcal L_b} (\mathcal K(a))$ of degree $m_b$}, that is, of the same degree at $\beta$. Further suppose that {$a^{-1}\tilde\beta(a)\in\mathcal L_a^\times$}.
    Let $b$ be nonzero algebraic element over {$\mathcal L_b$} such that $b^{{ m_b}}\in\mathcal L_b^\times$ {and $b^m\in\mathcal F^\times$ where $m\colonequals m_a m_b=(\mathcal K:\mathcal F)$}.
    
    Consider the $\mathcal F$-algebra
    \begin{equation*}
        {\mathfrak A\colonequals \sum_{j=0}^{{m_b}-1} \mathcal K(a) \cdot b^j,}
    \end{equation*}
    where conjugation by $b$ acts as $\tilde\beta$ on $\mathcal K(a)$.
    Viewing $\mathcal K(a)$ as a subalgebra of $\mathfrak A$, suppose that the elements $1,b,\ldots,b^{m_b-1}$ are linearly independent over $\mathcal K(a)$.
    
    {Then the following hold.
    \begin{lemmalist}
        \item We have that $\mathfrak A$ is an associative $\mathcal F$-algebra, and
        \begin{equation*}
        \mathfrak A= \bigoplus_{j=0}^{{m_b}-1} \mathcal K(a) \cdot b^j=  \bigoplus_{j=0}^{{m_b}-1} \bigoplus_{\ell=0}^{{m_a}-1} \mathcal K \cdot a^\ell \cdot b^j.
        \end{equation*}
        \item We have $(ab)^{m}=N_{\mathcal L_a/\mathcal F}\left(a^{{m_a}}\right) b^{m}\in\mathcal F^\times$.
        \item The subring $\mathfrak B\colonequals \sum_{k=0}^{m-1} \mathcal K (ab)^k$ of $\mathfrak A$ is a cyclic algebra, namely
        \[\mathfrak B=(\mathcal K/\mathcal L, \alpha\beta, (ab)^m)=\bigoplus_{k=0}^{m-1} \mathcal K(ab)^k.\] 
        \item We have $\mathfrak A=\mathfrak B$.
    \end{lemmalist}}
\end{lemma}

The following diagram shows the algebras occurring in the statement and the containments between them.
\[\begin{tikzcd}[ampersand replacement=\&,every arrow/.style={draw,no head},column sep=3mm,cells={nodes={minimum width=3em}}]
	\&\& {\mathfrak A} \& b \\
	a \& {\mathcal K(a)} \&\& \\ 
	{} \&\& {\mathcal K} \\
	{a^{{m_a}}} \& {\mathcal L_a} \&\& {\mathcal L_b} \& {b^{{m_b}}} \\
	\&\& {\mathcal F} \& b^m
	\arrow[from=1-3, to=2-2]
	\arrow["\in"{description}, draw=none, from=2-1, to=2-2]
	\arrow[from=2-2, to=3-3]
	\arrow["{\langle\alpha\rangle}"', "{m_a}" near end, from=3-3, to=4-2]
	\arrow["{\langle\beta\rangle}", "{m_b}"' near end, from=3-3, to=4-4]
	\arrow["\in"{description}, draw=none, from=4-1, to=4-2]
	\arrow[""',from=4-2, to=5-3]
	\arrow["",from=4-4, to=5-3]
	\arrow["\ni"{description}, draw=none, from=4-5, to=4-4]
	\arrow["\lrcorner"{anchor=center, pos=0.125, rotate=135}, draw=none, from=5-3, to=3-3]
    \arrow["\ni"{description}, draw=none, from=5-3, to=5-4]
    \arrow["\ni"{description}, draw=none, from=1-3, to=1-4]
\end{tikzcd}\]

\begin{proof}
    {Assertion (i) is immediate from the assumptions.}
    {To verify (ii), we expand the product by gathering the $b$-terms to the right.
	\begin{align*}
		\left(ab\right)^{ m} &= \left(a\cdot \tilde\beta(a)\cdots \tilde\beta^{{m_b}-1}(a)\right)^{{m_a}} b^{ m} & \text{$b$ acts via $\tilde\beta$} \\
		&= \left(a^{{m_a}}\cdot \tilde\beta\left(a^{{m_a}}\right)\cdots \tilde\beta^{{m_b}-1}\left(a^{{m_a}}\right)\right) b^{ m} \\
		&= N_{\mathcal L_a/\mathcal F}\left(a^{{m_a}}\right) b^{ m} & \langle\tilde\beta|_{\mathcal L_a}\rangle=\Gal(\mathcal L_a/\mathcal F)
	\end{align*}}

    {We proceed to show (iv). From the definition, it is clear that $\mathfrak B\subseteq \mathfrak A$. By applying the assumption $a^{-1}\tilde \beta(a)\in\mathcal L_a$, we see that for all $k\ge0$, $a^{-1}\tilde\beta^k(a)\in\mathcal L_a^\times$, and there is a unit $u_k\in\mathcal L_a^\times$ such that $(ab)^k=u_k a^k b^k$.}
    {With $k=m_a$, we obtain $(ab)^{m_a}=u_{m_a} a^{m_a} b^{m_a}\in \mathfrak B$. Since $u_{m_a}\in\mathcal L_a^\times$ and $a^{m_a}\in\mathcal L_a^\times\subseteq \mathcal K\subseteq \mathfrak B$, we deduce that $b^{m_a}\in\mathfrak B$. By coprimality of $m_a$ and $m_b$, we can write $m_a t=1+m_b s$ for some $t,s\in\ZZ$. Then $b^{m_at}=b\cdot b^{m_b s}$, and since $b^{m_b}\in\mathcal K^\times$, we have $b\in\mathfrak B$. Since $ab\in\mathfrak B$ by definition, this also implies $a\in\mathfrak B$, so $a^\ell b^j\in\mathfrak B$ for all $0\le \ell\le m_a-1$ and $0\le j\le m_b-1$, which shows $\mathfrak A\subseteq \mathfrak B$, thereby concluding the proof of (iv).}

    {Finally we show (iii). Using (iv) and the assumed $\mathcal K(a)$-linear independence of $1,b,\ldots,b^{m_b-1}$, we have $\dim_{\mathcal K} \mathfrak B=\dim_{\mathcal K}\mathfrak A=m$, whence the sum in the definition of $\mathfrak B$ is direct. For all $x\in \mathcal K$, we have $ab\cdot x=a \cdot \beta(x) b=(\alpha\beta)(x)\cdot ab$, which shows that $\mathfrak B$ possesses the claimed cyclic algebra structure.}
\end{proof}

{We now return to the double sum in \eqref{eq:A-double}.}

\begin{corollary} \label{A-cyclic}
	The algebra $A$ has the following cyclic algebra description:
	\begin{equation*}
		A=\bigoplus_{k=0}^{\frac{w_\chi}{\vf}\sFeta-1} \Q^{ F(\eta)(\omega)}\left(\Gamma^{w_\chi}\right)\left(\pi_{D_\eta} \gamma^{\vf}\right)^k  .
	\end{equation*}
	Conjugation by $\pi_{D_\eta} \gamma^{\vf}$ acts as $\sigma\tauf$ {on $\Q^{ F(\eta)(\omega)}\left(\Gamma^{w_\chi}\right)$}, and \[\left(\pi_{D_\eta} \gamma^{\vf}\right)^{\frac{w_\chi}{\vf} \sFeta}=N_{ F(\eta)/ F_\chi}(\pi_\eta)\gamma^{w_\chi \sFeta}  .\]
\end{corollary}
\begin{proof}
    {Apply \cref{combining-double-sum} with $\mathcal K\colonequals\Q^{F(\eta)(\omega)}(\Gamma^{w_\chi})$, $\mathcal F\colonequals\Q^{F_\chi}(\Gamma^{w_\chi})$, $\mathcal L_a\colonequals\Q^{F(\eta)}(\Gamma^{w_\chi})$, $\mathcal L_b\colonequals\Q^{F_\chi(\omega)}(\Gamma^{w_\chi})$, $a\colonequals\pi_{D_\eta}$, $b\colonequals\gamma^{v_\chi}$, $\alpha\colonequals\sigma$, and $\beta\colonequals\tau$.
    The degrees are $m_a=s_\eta$ and $m_b=w_\chi/v_\chi$: the former divides $p-1$ and the latter is a $p$-power, so they are coprime. Existence of $\tilde\beta$ is \cref{extending-tau-to-D}, which acts as multiplication by a unit due to \cref{extending-tau-non-uniquely}. For the second displayed formula, recall that $\pi_{D_\eta}^{s_\eta}=\pi_\eta$.}
\end{proof}

The fixed field of the automorphism $\sigma\tauf$ of $ F(\eta)(\omega)$ is $ F_\chi$, {and $\Gal(F(\eta)(\omega)/F_\chi)=\langle(\sigma\tauf\rangle|_{F(\eta)(\omega)})$. Noting that $\Gal(\Q^{ F(\eta)(\omega)}\left(\Gamma^{w_\chi}\right)/\Q^{ F_\chi}\left(\Gamma^{w_\chi}\right))\simeq\Gal(F(\eta)(\omega)/F_\chi)$, it follows from \cref{A-cyclic} that} the centre of $A$ is $\Q^{ F_\chi}\left(\Gamma^{w_\chi}\right)$. Therefore in the usual notation of cyclic algebras, \cref{A-cyclic} means that $A$ is of the form
\[A=\left(\Q^{ F(\eta)(\omega)}\left(\Gamma^{w_\chi}\right) \Big/ \Q^{ F_\chi}\left(\Gamma^{w_\chi}\right), \sigma\tauf, N_{ F(\eta)/ F_\chi}(\pi_\eta)\cdot\gamma^{w_\chi \sFeta}\right)  .\]

Fix an isomorphism $\ZZ_p\llb \Gamma^{w_\chi}\rrb\simeq \ZZ_p[[T]]$ with $\gamma^{w_\chi}$ corresponding to $(1+T)$. Under this isomorphism, we have
\begin{align*}
	\Q^{ F(\eta)(\omega)}(\Gamma^{w_\chi})&\simeq\Frac(\OO_{ F(\eta)(\omega)}[[T]])\equalscolon \KKK  , \\
	\Q^{ F_\chi}(\Gamma^{w_\chi})&\simeq\Frac(\OO_{ F_\chi}[[T]])\equalscolon \FFF  .
\end{align*}
The element $N_{ F(\eta)/ F_\chi}(\pi_\eta)\cdot\gamma^{w_\chi \sFeta} \in \Q^{ F_\chi}(\Gamma^{w_\chi})$ is then identified with 
\[a\colonequals N_{ F(\eta)/ F_\chi}(\pi_\eta)(1+T)^{\sFeta} \in \FFF  .\]

Wedderburn's theorem provides a sufficient (but not necessary) condition as to whether a cyclic algebra is a skew field, see \cite[(30.7)]{MO} or \cite[(14.9)]{Lam}. In this case, it states that $A$ is a skew field whenever $a$ has order $\frac{w_\chi}{\vf}\sFeta$ in the norm factor group $\FFF^\times / N_{\KKK/\FFF}(\KKK^\times)$. Since $\sFeta$ and $w_\chi/\vf$ are coprime, in order to be able to apply Wedderburn's theorem, it suffices to show that the order of $a$ is divisible by both of them; {the converse divisibility follows from Artin reciprocity.}

\begin{lemma}
	The order of $a$ in the norm factor group is divisible by $\sFeta$.
\end{lemma}
\begin{proof}
	Recall the augmentation exact sequence:
	\begin{align*}
		0\to T \OO_{ F(\eta)(\omega)}[[T]]\to \OO_{ F(\eta)(\omega)}[[T]] &\xrightarrow{\aug} \OO_{ F(\eta)(\omega)}\to 0\\
		T&\xmapsto{\phantom{\aug}} 0
	\end{align*}
	Localising at the kernel, which is a prime ideal, allows us to extend the augmentation map to 
	\[\aug: \OO_{ F(\eta)(\omega)}[[T]]_{(T)} \to  F(\eta)(\omega)  .\]
	As the Galois action is trivial on $T$, the augmentation map is compatible with the norm maps, that is, for all $x\in\OO_{ F(\eta)(\omega)}[[T]]_{(T)}$,
	\[\aug\left(N_{\KKK/\FFF}(x)\right)=N_{ F(\eta)(\omega)/ F_\chi}(\aug(x))  .\]
	Suppose that there is an $\alpha\in\KKK$ such that $N_{\KKK/\FFF}(\alpha)=a^i$. Since $T\nmid a^i$, such an element $\alpha$ is in fact contained in $\OO_{ F(\eta)(\omega)}[[T]]_{(T)}$. Thus the augmentation map is defined on $\alpha$, and 
	\begin{align} 
		N_{ F(\eta)(\omega)/ F_\chi}(\aug(\alpha))&=\aug(a^i)  . \notag\\
		\intertext{Using the definition of $a$ as well as transitivity of norms in a tower of extensions, we get}
		N_{ F(\eta)/ F_\chi}\left(N_{ F(\eta)(\omega)/{F(\eta)}}(\aug(\alpha))\right)&=N_{ F(\eta)/ F_\chi}(\pi_\eta^i)  . \label{eq:aug-alpha}
	\end{align}
	In particular, the two sides of \eqref{eq:aug-alpha} have the same valuation in $ F_\chi$:
	\[\ord_{\pi_\eta}N_{ F(\eta)(\omega)/{F(\eta)}}(\aug(\alpha)) =i  .\]
	The extension $ F(\eta)(\omega)/ F(\eta)$ is unramified of degree $\sFeta$, therefore $\ord_{\pi_\eta} \circ N_{ F(\eta)(\omega)/\QQ_{p}( \eta)}$ has image in $\sFeta\ZZ$. In particular, $\sFeta\mid i$.
\end{proof}

Therefore $A$ is a skew field if (but not necessarily only if) $w_\chi/\vf$ also divides the order of $a$. The extensions $ F(\eta)/ F_\chi$ and $ F_\chi(\omega)/ F_\chi$ have coprime degrees $w_\chi/\vf$ resp. $\sFeta$, and their compositum is $ F(\eta)(\omega)$. {Let $\LLL\colonequals \Frac(\OO_{ F(\eta)}[[T]])$. 
Then by transitivity of norm maps, there is a canonical surjective homomorphism
\[\FFF^\times\big/N_{\KKK/\FFF}(\KKK^\times) = \FFF^\times\big/N_{\LLL/\FFF}\left(N_{\KKK/\LLL}(\KKK^\times)\right) \twoheadrightarrow \FFF^\times\big/N_{\LLL/\FFF}(\LLL^\times) . \]
It follows that} $a$ has order divisible by $w_\chi/\vf$ in $\FFF^\times/N_{\KKK/\FFF}(\KKK^\times)$ if its image has order divisible by $w_\chi/\vf$ in $\FFF^\times/N_{\LLL/\FFF}(\LLL^\times)$.

\begin{lemma} \label{totally-ramified-order}
	Suppose that $ F(\eta)/ F_\chi$ is totally ramified.
	Then the order of $a$ in the norm factor group is divisible by $w_\chi/\vf$.
\end{lemma}
\begin{proof}
	In the norm factor group $\FFF^\times/N_{\LLL/\FFF}(\LLL^\times)$, the class of $a$ {agrees} with that of $(1+T)^{\sFeta}$.
	So suppose that $(1+T)^{\sFeta i}$ is the norm of some $\alpha(T)\in\LLL$. 
	
	We first show that without loss of generality, we may assume $\alpha(T)\in\OO_{ F(\eta)}[[T]]$. Indeed, the Weierstraß preparation theorem allows us to write 
	\[\alpha(T)=\pi_\eta^\ell \cdot\frac{F(T)}{G(T)}\]
	where $\ell\in\ZZ$ and $F(T)\in \OO_{ F(\eta)}[[T]]$ and 
	\[G(T)=\prod_{j=1}^k P_j(T)\in \OO_{ F(\eta)}[T] \] 
	is a product of distinguished irreducible polynomials $P_j(T)$, {not necessarily distinct}. The norm of $\alpha(T)$ is the product of its Galois conjugates; it follows that $\ell=0$. Moreover, since $(1+T)^{\sFeta i}$ has no denominator and since each $P_j(T)$ is irreducible, for each $j$ there exists a Galois conjugate $\tilde P_j(T)$ of $P_j(T)$ such that {$\prod_{j=1}^k\tilde P_j(T)\mid F(T)$}. Then 
	\[\alpha(T)\cdot\prod_{j=1}^k \frac{P_j(T)}{\tilde P_j(T)}\in\OO_{ F(\eta)}[[T]]\]
	has the same norm as $\alpha(T)$. From now on, we assume that $\alpha(T)\in\OO_{ F(\eta)}[[T]]$.
	
	Since $\alpha$ has integral coefficients, it is convergent at every element $x$ of the maximal ideal $\mm_\chi$ of $ F_\chi$. Moreover, since the Galois action is trivial on $T$ as well as on $ F_\chi$, for all $x \in \mm_\chi$ we have
	\begin{equation} \label{eq:N-seta}
		N_{\LLL/\FFF}(\alpha(T))\big|_{T=x}=N_{ F(\eta)/ F_\chi}(\alpha(x))=(1+x)^{\sFeta i}  .
	\end{equation}
	
	For totally ramified cyclic extensions of local fields, all Tate cohomology groups of the unit group are cyclic of order equal to the degree, see \cite[Corollary~2.11]{EN}. In particular,
	\begin{equation} \label{eq:EN-Tate-0}
		\hat H^{0}\left(\Gal( F(\eta)/ F_\chi), \OO_{ F(\eta)}^\times\right)\simeq \ZZ\big/\textstyle\frac{w_\chi}{\vf}\ZZ  .
	\end{equation}
	The left hand side is the unit norm factor group
	$\OO_{ F_\chi}^\times \big/ N_{ F(\eta)/ F_\chi} \OO_{ F(\eta)}^\times$. If $\#\overline{ F_\chi}$ denotes the order of the residue field of $ F_\chi$, then there is a decomposition
	\[\OO_{ F_\chi}^\times\simeq\mu_{\#\overline{ F_{\chi}}-1}\times U^1_{ F_\chi}  .\] 
	The group of roots of unity here has order coprime to $p$, while $w_\chi/\vf$ is a $p$-power. Therefore the isomorphism \eqref{eq:EN-Tate-0} on the unit norm factor group descends to principal units:
	\[U^1_{ F_\chi} \Big/ N_{ F(\eta)/ F_\chi}\left(U^1_{ F(\eta)}\right)\simeq \ZZ/\textstyle\frac{w_\chi}{\vf}\ZZ  .\]
	Let $u\in U^1_{ F_\chi}$ be a principal unit whose image in this factor group is a generator. Then $u^{\sFeta}$ is also a generator, since $\sFeta$ is coprime to $p$. Evaluating \eqref{eq:N-seta} at $T\colonequals u-1$, we get that $u^{\sFeta i}$ is a norm. Therefore $w_\chi/\vf\mid i$, as was to be shown.
\end{proof}
This concludes the proof of \cref{fQGf-skew-field}. \qed
\begin{remark} \label{rmk:total-ramification-technical}
	The argument in the proof of \cref{totally-ramified-order} does not generalise to not necessarily totally ramified extensions, because \eqref{eq:EN-Tate-0} may fail. {An explicit example in which $f_\eta^{(j)}$ fails to be indecomposable in $\Q^F(\G)$ is described in \cite[Example~2.4.12]{thesis}.}
\end{remark}

\subsection{Proof of \texorpdfstring{\cref{dim-Schur}}{Theorem~\ref{dim-Schur}} in the totally ramified case}
\label{sec:dimSchur}
In the proof of \cite[Corollary~1.13]{NickelConductor}, Nickel showed that
\[w_\chi \nFeta \sFeta  = w_\chi \eta(1) = \chi(1) = \nFchi \sFchi \,\big|\, \nFchi \sFeta {(F(\eta):F{\chi})}= \nFchi \sFeta \frac{w_\chi}{\vf}  , \]
where the last equality is due to \cref{Qpeta-Qpchi-cyclic}. Therefore $\nFeta \vf\mid \nFchi$; in particular, we have an inequality $\vf \nFeta\le \nFchi$. We now show that this is sharp.

On the one hand, we may express the primitive central idempotent $\epsilon_\chi$ of $\Q^ F(\G)$ as
\begin{equation} \label{eq:epsilon-chi-decomp}
	\epsilon_\chi =\sum_{i=0}^{\vf-1} \epsilon(\eta_{i})=\sum_{i=0}^{\vf-1} \sum_{j=1}^{n_{\eta_{i}}} f_{\eta_{i}}^{(j)}  ,
\end{equation}
where the $\eta_{i}$s are as in \eqref{eta-i}, and we used that $\nu^{ F}_\chi=\vf$, see \cref{Qpeta-Qpchi-cyclic}.
The skew field $D_{\eta_{i}}$ has centre $ F(\eta_{i})= F(\eta)$, which is the same for every $i$ (see \cref{sec:preliminaries}).
The $\eta_{i}$s are all $\G$-conjugates of one another, see \eqref{eta-i}, so in particular, the dimensions $\eta_{i}(1)$ all agree, and so do the Schur indices $s_{\eta_{i}}=\sFeta$.
Therefore
\[\dim_{ F(\eta)} M_{n_{\eta_{i}}}(D_{\eta_{i}}) = s_{\eta_{i}}^2 n_{\eta_{i}}^2=\eta_{i}(1)^2=\eta(1)^2=\sFeta^2 \nFeta^2  . \]
It follows that $n_{\eta_{i}}=\nFeta$ is the same for all $i$.
So \eqref{eq:epsilon-chi-decomp} is an expression of $\epsilon_\chi$ as a sum of $\vf \nFeta$ idempotents. These are indecomposable by \cref{fQGf-skew-field}. Equivalently, all right ideals $f_{\eta}^{(j)} \Q^ F(\G)$ are simple right modules by \cite[(3.18.iii)]{CR}. This gives rise to a strictly descending chain of submodules of $\Q^ F(\G)\epsilon_\chi$ of length $\vf \nFeta$, with the factor modules being simple: in other words, this is a composition series for $\Q^ F(\G)\epsilon_\chi$.

On the other hand, 
\begin{equation*}
	\epsilon_\chi=\sum_{i=1}^{\nFchi} f_\chi^{(j)}
\end{equation*} 
is another decomposition of $\epsilon_\chi$. The idempotents $f_{\chi}^{(j)}$ are indecomposable in $\Q^ F(\G)$, so there is a composition series of length $\nFchi$.
Since any two composition series have the same length by \cite[(3.9)]{CR}, we obtain
\begin{equation} \label{eq:s-neta-nchi}
	\vf \nFeta=\nFchi  .
\end{equation}

The assertion about the Schur indices follows readily:
\begin{align*}
	\sFchi &= \frac{w_\chi \sFeta \nFeta}{\nFchi} & \text{shown in the proof of \cite[Corollary 1.13]{NickelConductor}} \\
	&= \frac{w_\chi \sFeta}{\vf} & \text{by \eqref{eq:s-neta-nchi}} \\
	&= ( F(\eta): F_{\chi}) \sFeta & \text{by \eqref{wchi-s}}
\end{align*}
This finishes the proof of \cref{dim-Schur} in the totally ramified case. \qed

\subsection{Preparations for the proof of \texorpdfstring{\cref{D-conj}}{Theorem~\ref{D-conj}}}

{Let $A_\chi$ be the algebra
\[A_\chi\colonequals \bigoplus_{\substack{i=0 \\ \vf \mid i}}^{p^{n_0}-1} \tD_\eta \cdot (\gcp)^{i/\vf}=\bigoplus_{\substack{i=0}}^{\frac{p^{n_0}}{\vf}-1} \tD_\eta \cdot (\gcp)^{i},\]
with conjugation by $\gcp$ acting as $\tauf$ on $\tD_\eta$. Here the second sum is simply a rewriting of the first one.}

{We will now give a further description of the algebra $A_\chi$; note that our results on $A_\chi$ won't use the totally ramified assumption. In the next section, we will proceed to show that $D_\chi$ is isomorphic to $A_\chi$ under the total ramification hypothesis.}

\begin{lemma} \label{Achi-centre}
    {The algebra $A_\chi$ has Schur index $s_\eta \cdot w_\chi/v_\chi$ and centre}
    \[{\cent(A_\chi)=\Q^{ F_{\chi}}\left((\Gcp)^{w_\chi/\vf}\right).}\]
\end{lemma}
\begin{proof}
    The element $(\gcp)^{w_\chi/\vf}$ is central in $A_{\chi}$: {this follows from \cref{gcp-conjugation} and the fact that $\tauf$ has order $w_\chi/\vf$. The subfield $F_{\chi}$ is central in $A_{\chi}$ because its elements commute with all elements of $\tD_\eta$ (it even lies in the centre of $D_\eta$) as well as with $\gcp$ (because of \cref{gcp-conjugation} and the fact that $F_\chi$ is the fixed field of $\tauf$).} It follows that $\Q^{ F_{\chi}}\left((\Gcp)^{w_\chi/\vf}\right)$ is a central subfield in $A_{\chi}$, that is,
	\begin{equation*}
		\Q^{ F_{\chi}}\left((\Gcp)^{w_\chi/\vf}\right) \subseteq \cent(A_{\chi}).
	\end{equation*}

    {For the converse, let 
    \[z=\sum_{\substack{i=0 \\ \vf \mid i}}^{p^{n_0}-1} z_i (\gcp)^{i/v_\chi} \in \cent(A_\chi)\]
    be a central element, where $z_i\in\tD_\eta$ for all $i$ in the summation range. Then $z$ commutes with $\gcp$, conjugation by which acts as $\tau$; since $\gamma_0$ is central and the elements in $D_\eta$ fixed by $\tau$ are precisely those in $F_\chi$, we conclude that $z_i\in\Q^{F_\chi}(\Gamma_0)$ for all $i$. Let $d\in D_\eta^\times$ be an element such that $\tau^j(d)\ne d$ for all $1\le j<w_\chi/v_\chi$: such an element exists because $\tau$ has order $w_\chi/v_\chi$. The central element $z$ also commutes with this $d$: then $d z d^{-1}=z$ forces $dz_i (\gcp)^{i/v_\chi} d^{-1}=d \tau^{i/v_\chi}(d^{-1}) z_i (\gcp)^{i/v_\chi}$ for each $i$. Since $d=\tau^{i/v_\chi}(d)$ if and only if $w_\chi\mid i$, this shows $z_i=0$ unless $w_\chi \mid i$. Finally, recall that $(\gcp)^{p^{n_0}/\vf}$ differs from $\gamma_0$ by a central unit $\af_{p^{n_0}/\vf}$, as can be seen from applying \cref{power-of-gcp} with $j\colonequals p^{n_0}/\vf$ and \eqref{eq:a-pn0-central}. The claim on $\cent(A_\chi)$ is now established.}

    {The dimension of $A_{\chi}$ as a $\Q^{F_{\chi}}\left((\Gcp)^{w_\chi/\vf}\right)$-vector space is as follows:}
	\begin{align*}
		{\dim_{\Q^{F_{\chi}}\left((\Gcp)^{w_\chi/\vf}\right)} (A_{\chi})}
        &= {\frac{\dim_{\Q^{F_{\chi}}\left((\Gcp)^{p^{n_0}/\vf}\right)} \left(A_{\chi}\right)}{\dim_{\Q^{F_{\chi}}\left((\Gcp)^{w_\chi/\vf}\right)} \left(A_{\chi}\right)}
        = \frac{w_\chi}{p^{n_0}} \cdot \dim_{\Q^{F_{\chi}}\left((\Gcp)^{p^{n_0}/\vf}\right)} \left(A_{\chi}\right)} \\
        &= {\frac{w_\chi}{p^{n_0}}\cdot \frac{p^{n_0}}{v_\chi} \cdot \dim_{\Q^{F_{\chi}}\left((\Gcp)^{p^{n_0}/\vf}\right)} \left(\tD_\eta\right) }\\
        &\mathop{=}^{(\dagger)} {\frac{w_\chi}{v_\chi} \cdot \dim_{\Q^{F_{\chi}}\left((\Gcp)^{p^{n_0}/\vf}\right)} \left(\Q^{F_{\chi}}\left((\Gcp)^{p^{n_0}/\vf}\right)\otimes_{F_\chi} D_\eta\right)} \\
        &= {\frac{w_\chi}{v_\chi} \cdot \dim_{F_\chi} \left(D_\eta\right)
        = \left(\frac{w_\chi}{v_\chi} \cdot s_\eta\right)^2.}
	\end{align*}
	{The step marked $(\dagger)$ is uses the relationship between $\gamma_0$ and $(\gcp)^{p^{n_0}/\vf}$ explained above, as well as $\tD_\eta=\Q^{F(\eta)}(\Gamma_0)\otimes_{F(\eta)} D_\eta=\Q^{F_\chi}(\Gamma_0)\otimes_{F_\chi} D_\eta$. The Schur index of $A_\chi$ is the square root of this dimension.}
\end{proof}

\begin{proposition} \label{Achi-skew-power-series}
    {The ring $A_\chi$ is isomorphic to the total ring of quotients of a skew power series ring, and its centre is isomorphic to the field of fractions of a power series ring, as described by the following commutative diagram:}
	\[
	\begin{tikzcd}
		A_\chi \ar[r, phantom, "="] &[-1em] \displaystyle\bigoplus_{\substack{i=0\\ \vf\mid i}}^{p^{n_0}-1}\tD_\eta \cdot (\gcp)^{i/\vf} \ar[r, "\sim"] & \Quot\Big(\OO_{D_\eta}[[X;\tauf,\tauf-\id]]\Big)=\Dchi \\ 
		\cent(A_\chi) \ar[u, hook] \ar[r,phantom,"="] & \Q^{ F_{\chi}}\left((\Gcp)^{w_\chi/\vf}\right) \ar[r,"\sim"] \ar[u,hook] & \Frac\left(\OO_{ F_{\chi}}[[(1+X)^{w_\chi/\vf}-1]]\right) = \cent(\Dchi) \ar[u,hook]
	\end{tikzcd}
	\]
	{The top horizontal map sends $\gcp\mapsto 1+X$, it is the identity on $\OO_{D_\eta}$, which extends $F$-linearly to $\tD_\eta$ by $\tD_\eta=\Q^F(\Gamma_0)\otimes_F D_\eta$.
	The bottom horizontal map is the identity on $\OO_{ F_{\chi}}$ and sends $(\gcp)^{w_\chi/\vf}\mapsto (1+X)^{w_\chi/\vf}$.}
\end{proposition}

\begin{proof}
    Since $\gcp$ is sent to $(1+X)$, the multiplication rule $\gcp d=\tauf(d)\gcp$ for $d\in D_\eta$ becomes $(1+X)d=\tauf(d)(1+X)$, which is equivalent to $Xd=\tauf(d)X+(\tauf-\id)(d)$: this is indeed the multiplication rule in $\OO_{D_\eta}[[X;\tauf,\tauf-\id]]$. 
	{As we have seen in \cref{Achi-centre}}, the centre is generated by $(\gcp)^{w_\chi/\vf}$; since $\gamma_0$ is central, the image of $\gamma_0$ is therefore determined by the image of $\gcp$. {More precisely, recall from \cref{power-of-gcp} that $\gamma_0=a_{p^{n_0}/v_\chi} (\gcp)^{p^{n_0}/v_\chi}$, and we have seen in \eqref{eq:a-pn0-central} that $\af_{p^{n_0}/\vf}$ is central in $\Q^ F(\Gamma_0)[H]\epsilon(\eta)$, so its counterpart $A_{p^{n_0}/\vf}$ under the Wedderburn isomorphism \eqref{eq:Wedderburn-eta} is central in $M_{n_\eta}(\tD_\eta)$, that is, $A_{p^{n_0}/v_\chi}=\bar A_{p^{n_0}/v_\chi} \cdot \mathbf 1_{n_\eta}$ is a scalar matrix where $\bar A_{p^{n_0}/v_\chi}\in\tD_\eta$. Therefore the top horizontal map in the diagram above sends $\gamma_0\mapsto \bar A_{p^{n_0}/v_\chi}\cdot (1+X)^{p^{n_0}/v_\chi}$.}
    
    We now have that the top horizontal map is well-defined and a ring homomorphism. The lower horizontal map is a well-defined isomorphism induced by the classical isomorphism between the Iwasawa algebra over $\OO_{ F_\chi}$ and the ring of formal power series over $\OO_{ F_\chi}$. Commutativity of the diagram follows directly from the definition of the arrows within.
	
	It remains to show that the top horizontal map is an isomorphism. {It is surjective, as the image contains all coefficients in $\OO_{D_\eta}$ as well as the variable $X$.} {By \cref{Achi-centre}, the crossed product algebra on the left is a left vector space over its centre, of dimension $(s_\eta \cdot w_\chi/v_\chi)^2$.} On the right hand side, the ring $\Quot(\OO_{D_\eta}[[X;\tauf,\tauf-\id]])$ is a skew field with centre $\Frac(\OO_{ F_{\chi}}[[(1+X)^{w_\chi/\vf}-1]])$ (\cref{full-centre-of-skew-power-series-ring}). 
	Since the dimensions agree, it follows that the top horizontal map must be an isomorphism.
\end{proof}

\subsection{Proof of \texorpdfstring{\cref{D-conj}}{Theorem~\ref{D-conj}} in the totally ramified case} \label{sec:D-conj-proof-totally-ramified}
\begin{proposition} \label{Dchi-fQGf}
	Let $\chi\in\Irr(\G)$ and let $\eta\mid\res_H^\G \chi$ be an irreducible constituent. Suppose that $ F(\eta)/ F_\chi$ is totally ramified. Then for all $1\le j\le \nFeta$, there is an isomorphism of rings
	\[\DFchi\simeq f_\eta^{(j)} \Q^ F(\G) f_\eta^{(j)}  .\]
	In particular, the right hand side is independent of the choice of $\eta$ and $j$.
\end{proposition}
\begin{proof}
	The $\chi$-part of the Wedderburn decomposition of $\Q^ F(\G)$ is $\epsilon_\chi\Q^ F(\G)\epsilon_\chi\simeq M_{\nFchi}(\DFchi)$. As noted in \eqref{eq:epsilon-chi-decomp}, we have $f_{\eta}^{(j)} \mid \epsilon_\chi$. Indecomposability of $f_{\eta}^{(j)}$ shows the claimed isomorphism.
\end{proof}

\begin{corollary} \label{Dchi-crossed-product}
	Suppose that $ F(\eta)/ F_\chi$ is totally ramified.
	{Then there is an isomorphism of rings
        \[\DFchi\simeq A_\chi=\bigoplus_{\substack{i=0 \\ \vf \mid i}}^{p^{n_0}-1} \tD_\eta \cdot (\gcp)^{i/\vf},\]
    with conjugation by $\gcp$ acting as $\tauf$ on $\tD_\eta$.}
\end{corollary}
\begin{proof}
	{We have the following chain of isomorphisms:}
	\begingroup\allowdisplaybreaks
	\begin{align*}
		\DFchi &\simeq f_\eta^{(1)} \Q^F(\G) f_\eta^{(1)} \simeq \bigoplus_{\substack{i=0 \\ \vf \mid i}}^{p^{n_0}-1} f_\eta^{(1)} \Q^F(\Gamma_0)[H] \gamma^i f_\eta^{(1)} \simeq \bigoplus_{\substack{i=0 \\ \vf \mid i}}^{p^{n_0}-1} f_\eta^{(1)} \Q^F(\Gamma_0)[H] f_\eta^{(1)} \cdot (\gcp)^{i/\vf} .
	\end{align*}
	\endgroup
    Here the first isomorphism is \cref{Dchi-fQGf}, the second one comes from \eqref{eq:Q-decomposition} and \cref{divisibility-by-v_chi}, the third one is \eqref{eq:yeta-higher} and \cref{power-of-gcp}. The claimed isomorphism now follows from \eqref{tDeta-feta}.
	As we have seen in \cref{gcp-conjugation}, conjugation by $\gcp$ acts as $\pt$ on $f_\eta^{(1)} \Q^ F(\Gamma_0)[H] f_\eta^{(1)}$, which, by definition of $\pt$, becomes the action of $\tauf$ on $\tD_\eta$. This proves the assertion. 
\end{proof}

\begin{proof}[Proof of \cref{D-conj}]
    Combine \cref{Achi-skew-power-series,Dchi-crossed-product}.
\end{proof}

\printbibliography
\end{document}